\documentclass{article}

\usepackage{amsmath,amsfonts,amssymb,amsthm,graphicx}

\def\d#1{{#1\kern-0.4em\char"16\kern-0.1em}}
\def\D#1{{\raise0.2ex\hbox{-}\kern-0.4em #1}}

\newtheorem{theorem}{Theorem}[section]

\newtheorem{proposition}{Proposition}[section]
\newtheorem{remark}{Remark}[section]
\newtheorem{lemma}{Lemma}[section]
\newtheorem{corollary}{Corollary}[section]
\theoremstyle{definition}\newtheorem{example}{Example}[section]

\def\N{{\mathcal{N}}}
\def\R{{\mathcal{R}}}

\def\D{{\cal{D}}}

\def\A{{\mathcal{A}}}

\def\CC{\mathbb{C}}
\def\NN{\mathbb{N}}

\newcommand{\mat}[4]{\left[\begin{array}{cc}#1 & #2 \\ #3 & #4 \\
	\end{array}\right]}

\author{Bogdan D. Djordjevi\'c\footnote{Mathematical Institute of the Serbian Academy of Sciences and Arts, Belgrade, Republic of Serbia.\qquad {\tt bogdan.djordjevic@turing.mi.sanu.ac.rs;\ bogdan.djordjevic93@gmail.com}}, Neboj\v sa \v C. Din\v ci\'c\footnote{University of Ni\v s, Faculty of Sciences and Mathematics, Ni\v s, Republic of Serbia.\qquad {\tt nebojsa.dincic@pmf.edu.rs;\ ndincic@hotmail.com}}, and Mihailo Djuri\'c\footnote{Student at University of Belgrade, Faculty of Mathematics, Belgrade, Republic of Serbia.\qquad {\tt mm25073@alas.matf.bg.ac.rs;\ mihajlodjuric26@gmail.com}}}
\date{}

\begin{document}
	
	\title{Commutator-based Solutions to $AXA=XAX$ when $A^n=tA^2$}	
	
	\maketitle
	
	\begin{abstract}
		This paper offers a novel analytical method for obtaining new non-commuting solutions for the Yang-Baxter--like matrix equation $AXA=XAX$, when the coefficient matrix $A$ satisfies the square-cyclic condition $A^n=tA^2$. All the consistency criteria are proved and the obtained results are demonstrated on worked examples. The findings provided by this article simultaneously generalize several existing special cases regarding this topic. Consequently, by employing the obtained results, the previously unstudied inhomogeneous Yang-Baxter matrix equation $AXA=XAX+B$ is solved when $A$ is either invertible, or satisfies $A^2=0$.
	\end{abstract}
	
	\section{Introduction} \label{intro-sec}
	
	The Yang-Baxter--like matrix equation 
	\begin{equation}\label{YBME}
		AXA=XAX
	\end{equation}
	appears in various areas of mathematics and physics: from modeling fluid dynamics, integrable scattering systems and quantum groups, to knot theory, braid groups, digital signal processing, and object-centric machine learning, to name a few. Its broad scope of applications is precisely the reason why this equation has been studied in many different settings, such as when $A$ is diagonalizable, nilpotent, idempotent, $m-$potent ($A^m=A$), permutation matrix, circulant matrix, rotation matrix, unitary matrix, and so on, consult \cite{DinDjor}, \cite{NCDBDD2}, \cite{BDperm}, \cite{Don2016}, \cite{DonDin2017}, \cite{QH}, \cite{MSIA}, \cite{MSIAJDQHLZ}, \cite{MSIAJDQ}, \cite{a4a}, \cite{ZYWDH2021}, and the references therein. Thus, solving the equation \eqref{YBME} under such concrete constraints relies on highly specialized techniques and on case-specific arguments. As a result, such methods are generally unsuitable for solving the equation in a broader context, since the associated techniques and calculations do not naturally extend to weaker assumptions. Consequently, the obtained solutions $X$ depend on the imposed particular structure of the input matrix $A$ as well, and do not offer an insight into a more general scenario.
	
	In this paper, we develop an original commutator-based approach that avoids these restrictive assumptions on the matrix $A$ and enables treatment of \eqref{YBME} under a weaker premise. For clarity's sake, we assume that the matrix $A$ is provided in the manner that 
	\begin{equation}
		\label{sqcA} A^n=tA^2, \; n>2, \; t\neq 0,
	\end{equation} and we refer to this condition as the square-cyclic property. Clearly \eqref{sqcA} contains all the special cases for $A$ mentioned above, but also contains plenty of new possibilities which, to the best of our knowledge, have not yet been studied in the literature (for example, when $A$ is skew-symmetric). The free nonzero parameter $t$ indicates that the norm of $A$ need not be preserved under the power transform $A\mapsto A^n$, thus leaving the partial isometry setting. Our calculations will rely on the square-cyclic property \eqref{sqcA}, however, we point  out that our principle and methodology hold for any polynomially-cyclic condition imposed on $A$.
	
	Throughout the paper, it is assumed that all matrices are complex, and are treated as linear operators over finite-dimensional complex vector spaces. Unless stated differently, it is assumed that $A\in\CC^{N\times N}$, for some natural $N>1$, and we seek the solution matrices $X$ within the same matrix algebra $\CC^{N\times N}$.
	
	\subsection{Preliminaries}

	For a given matrix/linear operator $L$, we use the notation $\N(L)$ and $\R(L)$ for the null-space and the range of $L$, respectively. Similarly, if $\lambda$ is an eigenvalue for $L$, then the corresponding eigenspace is denoted as $\N(L-\lambda I)$. The set of all eigenvalues, i.e., the spectrum of $L$ is denoted as $\sigma(L)$. 
	
	Recall that the Kronecker product of an $\ell\times r$ matrix $S$ and a $p\times q$ matrix $T$ is defined as the $p\ell\times qr$ block matrix $S\otimes T$ given by
	$$S\otimes T=\left[\begin{array}{cccc}
		s_{11}T & ... & s_{1 r}T\\
		... & ... & ...\\
		s_{\ell 1}T & ... & s_{\ell r}T
	\end{array}\right]=[s_{ij}T].$$
	
	It is well-known that the Kronecker product is associative and bilinear, but non-commutative. Closely related to the Kronecker product is the usual ``vectorization trick", often employed in solving matrix equations: by ``$\operatorname{vec}$" we denote the vectorization operator that stacks the columns of a given matrix $S=[s_{ij}]_{\ell\times r}$ one atop of the other:
	$$\operatorname{vec}(S)=[s_{11},...,s_{\ell 1},s_{12},...,s_{\ell 2},...,s_{1r},...,s_{\ell r}]^T.$$
	(We emphasize that in \cite{BiG} the vectorization is done by stacking the rows, so some formulae are slightly different). Then for any conformal matrices $S,T,L$:
	$$\operatorname{vec}(TLS)=(S^T\otimes T)\operatorname{vec}(L).$$
	
	Another method used for solving matrix equations involves the employment of generalized inverses. For a given matrix $S\in\mathbb{C}^{\ell\times r}$, a matrix $G\in\CC^{r\times\ell}$ that satisfies $SGS=S$ is said to be an inner inverse of $S$. The set of all inner inverses of $S$ is usually denoted by $S\{1\}$, while its elements are denoted by $S^-$ or by $S ^{(1)}$. Probably the most important application of inner inverses is the following theorem due to Penrose.
	
	\begin{theorem}\cite{Pen, BiG}\label{ThmPenrose}
		Let $S\in\mathbb{C}^{\ell\times r},\;T\in\mathbb{C}^{p\times q},\;W\in\mathbb{C}^{\ell\times q}$. Then the matrix equation 
		\begin{equation}\label{SLTW}SLT=W\end{equation}
		is consistent for $L$, if and only if, for some $S^{(1)},T^{(1)}$, 
		$$SS^{(1)}WT^{(1)}T=W,$$ 
		in which case the general solution is 
		$$L=S^{(1)}WT^{(1)}+Z-S^{(1)}SZTT^{(1)},$$
		for an arbitrary matrix $Z\in\mathbb{C}^{r\times p}$.
	\end{theorem}
	Specially, recall that there exists one and only one inner inverse $S^\dagger\in S\{1\}$ of $S$ which satisfies all four Penrose equations:
	$$(1)\;SS^\dagger S=S,\;(2)\;S^\dagger SS^\dagger=S^\dagger,\;(3)\;(SS^\dagger)^*=SS^\dagger,\;(4)\;(S^\dagger S)^*=S^\dagger S,$$
	where $M^*$ denotes the usual conjugate transpose of $M$. That unique inner inverse is called the Moore-Penrose inverse of $S$, and is traditionally denoted by $S^\dag$. If $S$ is nonsingular, then $S^\dag=S^{-1}$. For more on the generalized inverses of matrices we refer to the book \cite{BiG} and the references therein. In particular, note that for the Jordan matrix $J_2(0)$ we have
	$$J_2(0)^\dag=\left[\begin{array}{cc}
		0 & 1\\
		0 & 0
	\end{array}\right]^\dag = \left[\begin{array}{cc}
		0 & 0\\
		1 & 0
	\end{array}\right]=J_2(0)^T.$$

	We recall the following properties which will be called upon later:
	
	\begin{proposition}\label{Kron}
		Let $S,T,C,D$ be complex matrices of conformable dimensions. Then:
		\begin{itemize}
			\item[i)] $S\otimes (T+C)=S\otimes T+S\otimes C,\;(T+C)\otimes S=T\otimes S+C\otimes S$.
			\item[ii)] (mixed-product property) $(S\otimes T)(C\otimes D)=(SC)\otimes (TD)$.
			\item[iii)] (Moore-Penrose inverse of the Kronecker product) $(S\otimes T)^\dag = S^\dag \otimes T^\dag$.
		\end{itemize}
	\end{proposition}
	
	\section{Solutions with Respect to a Commutator}
	Recall that for the equation \eqref{YBME}, a solution $X$ is said to be commuting if $AX=XA$, and non-commuting otherwise. It has been noted that commuting solutions to YBME behave quite differently compared to the non-commuting ones, see e. g.  \cite{DinDjor},  \cite{NCDBDD2}, and numerous references therein. This remark inspired us to proceed with the approach presented in this paper: the main idea for solving \eqref{YBME} is to introduce another matrix  $Y\in\CC^{N\times N}$ and to seek those solutions which in addition satisfy the commutator relation
	\begin{equation}
		\label{Y} Y=AX-XA.
	\end{equation}
	Such solutions will be referred to as $Y-$based solutions, and they define one family in the solution set, where each member satisfies \eqref{YBME} and \eqref{Y}. Precisely, if
	$$\mathcal{S}_A=\{X\in\CC^{N\times N}: AXA=XAX\}$$
	then 
	$$\mathcal{S}_A(Y)=\{X\in\mathcal{S}_A: AX-XA=Y\}.$$
	Varying the commutator matrix $Y$ offers numerous families of parametric solutions to \eqref{YBME}, where each family corresponds to one specific commutator $Y$. Note that these families are not affine (nor linear) spaces, since the problem is non-additive and nonlinear, see \cite{DinDjor} and \cite{NCDBDD2}.
	
	We emphasize that, as shown in  \cite{DinDjor} and \cite{NCDBDD2}, one non-commuting solution gives rise to an entire class of non-commuting solutions. Indeed, if $X_0$ is a non-commuting solution to \eqref{YBME} then for any complex function $f$ which is holomorphic in a region that contains the eigenvalues of $A$ and has no zeros in $\sigma(A)$, it follows that $f(A)X_0f(A)^{-1}$ is also a non-commuting solution to \eqref{YBME}. Moreover, if holomorphic functions $f$ and $g$ (holomorphic in a suitable region that contains the eigenvalues of $A$), satisfy $f(z)g(z)=1$ in $\sigma(A)$, then $f(A)X_0g(A)$ also provides a new non-commuting solution. Therefore, finding one non-commuting solution can be thought of as finding the initial step (the ``intellectual guess'') of a certain procedure, which will generate entire classes of non-commuting solutions. This is the main reason why we are primarily focused on providing (at least one) non-commuting solution. With this in mind, the next result immediately follows:
	\begin{corollary}
		Let $X_0$ be a non-commuting solution to \eqref{YBME}, and let $Y_0:=AX_0-X_0A$. The following claims are true:
		\begin{itemize}
			\item[(a)] If $f$ is a complex function which is holomorphic in a region that contains $\sigma(A)$ and which has no zeros in $\sigma(A)$, then
			$$f(A)X_0f(A)^{-1}\in\mathcal{S}_A\left(f(A)^{-1}Y_0f(A)\right).$$
			\item[(b)]  If $f$ and $g$ are holomorphic functions in a region that contains the eigenvalues of $A$ and satisfy $f(z)g(z)=1$ in $\sigma(A)$, then 
			$$f(A)X_0g(A)\in\mathcal{S}_A\left(g(A)Y_0f(A)\right).$$
		\end{itemize}
	\end{corollary}

	Surprisingly enough, while all the solutions obtained within this manuscript will be expressed in terms of some $Y,$ throughout the manuscript it is not imperative to assume that $Y\neq0$. In other words, all calculations conducted within this article automatically hold even when $Y=0$, although this case is not that interesting given that all commuting solutions to \eqref{YBME} are accounted for (see Subsection \ref{commuting} below). 
	
	However, the commutator $Y$ cannot be chosen completely arbitrarily: the identity operator, $Y=I$, is \emph{never} a commutator of two matrices. Moreover, while commuting solutions always exist, the existence of non-commuting solutions to YBME is not a given (see \cite{DinDjor} and \cite{NCDBDD2}). Therefore, solving \eqref{YBME}--\eqref{Y} simultaneously is not always possible. Some necessary consistency conditions are provided by the following lemma:
	\begin{lemma}[Necessary conditions for $Y$]\label{necesY}
		Let $A$ be a square matrix such that $A^n=tA^2$ holds for some $n>2$ and $t\neq0$. Let $X$ and $Y$ be provided such that \eqref{Y} is satisfied. Then the following equality holds:
		\begin{equation}
			\label{tY}tAYA=\sum_{s=1}^{n-1}A^sYA^{n-s}.
		\end{equation}
	\end{lemma}
	\begin{proof}
		From \eqref{Y} we have
		\begin{eqnarray}
			\label{AYA}
			\begin{aligned}
				AX-XA=Y &\Rightarrow A^2XA-AXA^2=AYA\\
				&\Leftrightarrow A(AXA)-(AXA)A=AYA.
			\end{aligned}
		\end{eqnarray}
		But then it follows, by a standard commutator expansion, that for any $m\in\NN$:
		\begin{equation}
			\label{mSylAXA}
			A^m(AXA)-(AXA)A^m=\sum_{s=0}^{m-1}A^s\left(AYA\right)A^{m-1-s}.
		\end{equation}
		On the other hand, invoking $A^n=tA^2$ gives
		\begin{eqnarray}
			\label{tA2n}
			\begin{aligned}
				A^2XA-AXA^2=AYA\Leftrightarrow & A^nXA-AXA^n=tAYA\\
				\Leftrightarrow &A^{n-1}AXA-AXAA^{n-1}=tAYA.
			\end{aligned}
		\end{eqnarray}
		By choosing $m:=n-1$ we obtain from \eqref{mSylAXA} and \eqref{tA2n}:
		$$
		tAYA=\sum_{s=0}^{n-2}A^{s+1}YA^{n-(1+s)}=\sum_{s=1}^{n-1}A^sYA^{n-s}.
		$$
	\end{proof}

	\subsection{Reduced System of Matrix Equations}
	
	Throughout the paper we are working with the following:\\
	
	\noindent\textbf{Assumption} $\left(\mathcal{A}_1\right)$: The matrix $A$ is given such that \eqref{sqcA} holds, and a matrix $Y$ is fixed.\\
	
	\noindent\textbf{Problem} ($\mathcal{P}_1$): Under the assumption $(\mathcal{A}_1)$ find all (if any) matrices $X$ which solve \eqref{YBME} and subject to \eqref{Y}.\\
	
	In order to solve ($\mathcal{P}_1$) we start by rewriting the initial equation \eqref{YBME} in terms of a system of nonlinear equations which subject to some natural linear constraints. We utilize some very convenient tools from operator theory, generalized inverses and Sylvester equations. Given the square-cyclic condition \eqref{sqcA} that is imposed on the matrix $A$, it is natural to proceed with the following decomposition:
	
	\begin{theorem}[Core-Nilpotent Decomposition] \label{core-nilpotent}
		Let $A$ be a square matrix of size $N \times N$ with complex entires. If $A^n=tA^2$ holds, for some natural $n>2$ and $t\neq0$, then the space $\CC^N$ decomposes into
		\begin{equation}\label{CN} \CC^N=\N(A^2)\oplus\R(A^2)\end{equation}
		and there exists an invertible matrix $Q\in\mathbb{C}^{N\times N}$ such that
		\begin{equation}
			\label{cnA}
			Q^{-1}AQ = 
			\begin{bmatrix}
				A_1 & 0 \\
				0 & A_2
			\end{bmatrix}:\left[\begin{array}{c}\N(A^2)\\ \R(A^2)\end{array}\right]\to\left[\begin{array}{c}\N(A^2)\\ \R(A^2)\end{array}\right]
		\end{equation}
		where $N_1=\dim \N(A^2)$ and $A_1\in\mathbb{C}^{N_1\times N_1}$ is a nilpotent matrix in $\N(A^2)$ of the nilpotency index at most $2$, $N_2=\dim \R(A^2)$ and $A_2\in\mathbb{C}^{N_2\times N_2}$ is an invertible matrix in $\R(A^2)$ (of course, $N=N_1+N_2$). 
	\end{theorem}
	
	\begin{proof}
		The Core-Nilpotent Decomposition of any square matrix (see e.g. \cite{BiG}) states that there exists a positive integer $p\geq1$ which is equal to both the ascend and the descend of the operator $A$. Respectively, the space $\CC^N$ decomposes into the direct sum $\CC^N=\N(A^p)\oplus\R(A^p)$, and there exists an invertible matrix $Q$ such that
		$$Q^{-1}AQ=\mat{A_1}{0}{0}{A_2}:\left[\begin{array}{c}\N(A^p)\\\R(A^p)\end{array}\right]\to \left[\begin{array}{c}\N(A^p)\\\R(A^p)\end{array}\right],$$
		with $A_2$ being invertible in $\R(A^p)$, while $A_1$ is nilpotent in $\N(A^p)$ with the nilpotency index being at most $p$. On the other hand, from $A^n=tA^2$ it follows that
		$$\N(A)\subseteq\N(A^2)\subseteq\N(A^3)\subseteq\ldots\subseteq\N(A^n)=\N(tA^2)=\N(A^2),$$
		proving that $p\leq 2$.
	\end{proof}
	For the rest of the manuscript, the letters $N_1$ and $N_2$ will respectively denote the dimensions $N_1=\dim\N(A^2)$, and $N_2=\dim\R(A^2)$, where we point out that $N=N_1+N_2$.

	Note that the block-diagonal form of the nilpotent matrix $A_1$ (with $A_1^2=0$) may contain Jordan blocks $J_2(0)$ as well as $J_1(0)$. It is clear that at least one $J_2(0)$ block is present only if $A_1\neq 0$.
	
	It is straightforward to see that the matrix $A_2$ from \eqref{cnA} also satisfies the square-cyclic condition $A_2^n=tA_2^2$, and that $\sigma(A_2)=\sigma(A)\setminus\{0\}$. Since $A_2$ is invertible, the latter is equivalent to $A_2^{n-2}=t I$, where $I$ is the identity operator on $\R(A^2)$.

	By \cite{DinDjor}, we can without loss of generality rearrange the basis vectors in a such manner, that the matrix $A$ has the form \eqref{cnA}. With respect to \eqref{CN}, 
	we put
	\begin{equation}
		\label{cnX}
		X = \begin{bmatrix} X_1 & X_2\\ X_3 & X_4 \end{bmatrix}:\left[\begin{array}{c}\N(A^2)\\ \R(A^2)\end{array}\right]\to\left[\begin{array}{c}\N(A^2)\\ \R(A^2)\end{array}\right].\end{equation}
	Then $X$ solves \eqref{YBME} if and only if the following system is consistent for $X_1$, $X_2$, $X_3$ and $X_4$:
	\begin{equation}
		\label{system}
		\begin{cases}
			A_1 X_1 A_1 = X_1 A_1 X_1 + X_2 A_2 X_3\\
			A_1 X_2 A_2 = X_1 A_1 X_2 + X_2 A_2 X_4\\
			A_2 X_3 A_1 = X_3 A_1 X_1 + X_4 A_2 X_3\\
			A_2 X_4 A_2 = X_3 A_1 X_2 + X_4 A_2 X_4.
	\end{cases}\end{equation}  
	On the other hand, for any given matrix $Y$ we have
	\begin{equation}
		\label{Ymat}
		Y=\mat{Y_1}{Y_2}{Y_3}{Y_4}:\left[\begin{array}{c}\N(A^2)\\ \R(A^2)\end{array}\right]\to\left[\begin{array}{c}\N(A^2)\\ \R(A^2)\end{array}\right].
	\end{equation}  
	Thus solving Problem $(\mathcal{P}_1)$ boils down to solving the system \eqref{system} while maintaining the following matrix equality
	\begin{equation}
		\label{AXXAYmat}
		\mat{Y_1}{Y_2}{Y_3}{Y_4}=\mat{A_1X_1-X_1A_1}{A_1X_2-X_2A_2}{A_2X_3-X_3A_1}{A_2X_4-X_4A_2},\end{equation}  or, equivalently, solving Problem $(\mathcal{P}_1)$ reduces to solving \eqref{system} while simultaneously solving the system of Sylvester equations
	\begin{equation}
		\label{Sylsystem}
		\begin{cases}
			A_1X_1-X_1A_1=Y_1\\
			A_1X_2-X_2A_2=Y_2\\
			A_2X_3-X_3A_1=Y_3\\
			A_2X_4-X_4A_2=Y_4.
		\end{cases}
	\end{equation}
	This additional condition can simplify the problem in the following sense. The matrix $A_1$ is nilpotent ergo $\sigma(A_1)=\{0\}$. Similarly, the matrix $A_2$ is invertible and corresponds to the core part of the matrix $A$, therefore $\sigma(A_2)=\sigma(A)\setminus\{0\}$. In other words, $\sigma(A_1)\cap\sigma(A_2)=\emptyset$, and the second and the third equation in the system \eqref{Sylsystem} are regular Sylvester equations, meaning that for any given $Y_2$, $Y_3$, there exist unique $X_2$ and $X_3$ that satisfy these equalities respectively. Since $A_1^2=0$ it follows that the matrices $X_2$ and $X_3$ are respectively calculated using Theorem \ref{nil_inv_Syl} below.
	
	\begin{theorem}\label{nil_inv_Syl} Let $L\in\CC^{\ell\times\ell}$ and $T\in\CC^{r\times r}$ be complex square matrices, and let $D\in\CC^{r\times\ell}$ be a rectangular matrix. The following statements are true:
		\begin{itemize}
			\item[i)] Let $T$ be invertible and let $L$ be a nilpotent matrix with $L^2=0$. The Sylvester equation $TZ-ZL=D$ is consistent and its solution is
			$$Z=T^{-1}(D+T^{-1}DL).$$
			\item[ii)] Let $L$ be invertible and let $T$ be a nilpotent matrix with $T^2=0$. The Sylvester equation $TZ-ZL=D$ is consistent and its solution is
			$$Z=-(D+TDL^{-1})L^{-1}.$$
		\end{itemize}
	\end{theorem}
	\begin{proof}
		We prove only the first statement. The matrix $T$ is invertible, so the Sylvester equation is equivalent to $Z=T^{-1}(D+ZL)$. Further, we have:
		\begin{align*}
			Z &= T^{-1}(D+ZL)\\
			&= T^{-1}\left(D+T^{-1}(D+ZL)L\right)\\
			&=T^{-1}(D+T^{-1}DL),
		\end{align*}
		because of $L^2=0$.
	\end{proof}
	Returning to our problem, by Theorem \ref{nil_inv_Syl} we have
	\begin{equation}
		\label{X2series} X_2=-(Y_2+A_1Y_2A_2^{-1})A_2^{-1},
	\end{equation}
	\begin{equation}
		\label{X3series}
		X_3=A_2^{-1}(Y_3+A_2^{-1}Y_3A_1).
	\end{equation}
	For short, we denote these unique connections \eqref{X2series} and \eqref{X3series} between the matrices $X_2$ and $Y_2$, and respectively between $X_3$ an $Y_3$, by
	\begin{equation}
		\label{X2}X_2:=\widehat{Y_2},
	\end{equation}
	\begin{equation}
		\label{X3}X_3:=\widehat{Y_3}.
	\end{equation}
	Moreover, by exploiting the property \eqref{tY}, we see that some additional conditions are necessary for the matrices $Y_1$ and $Y_4$: 
	
	\begin{lemma}[Necessary conditions for $Y_1$ and $Y_4$]\label{necesY1Y4}
		Let $A$ be a square matrix such that $A_2 ^n=tA_2^2$ holds for some $n>2$ and $t\neq0$. Let $X$ and $Y$ be provided such that \eqref{Y} is satisfied. If the matrices $A$, $X$, and $Y$ are respectively decomposed as \eqref{cnA}, \eqref{cnX}, and \eqref{Ymat}, then the following equalities hold:
		\begin{equation}
			\label{A1Y1=Y1A1}
			A_1Y_1+Y_1A_1=0.
		\end{equation}
		\begin{equation}
			\label{A1Y1A1}
			A_1Y_1A_1=0.
		\end{equation} 
		\begin{equation}
			\label{tY2}tY_4=\sum_{s=0}^{n-2}A_2^sY_4A_2^{n-2-s},
		\end{equation}
		\begin{equation}
			\label{0tY2}0=Y_4+\sum_{s=1}^{n-3}A_2^{s}Y_4A_2^{-s}.
		\end{equation}
	\end{lemma}
	\begin{proof} The first equation follows immediately from 
		$$A_1X_1-X_1A_1=Y_1,$$
		then multiplying this by $A_1$ from the left or from the right, respectively, gives
		$$Y_1A_1=A_1X_1A_1-X_1A_1^2=A_1X_1A_1=-A_1(A_1X_1-X_1A_1)=-A_1Y_1.$$
		This proves \eqref{A1Y1=Y1A1}, and, consequently, \eqref{A1Y1A1} follows immediately. 
		
		On the other hand, we see that the conditions for Lemma \ref{necesY} are met, therefore \eqref{tY} holds. Rewriting $A$, $X$, and $Y$ in terms of \eqref{cnA}, \eqref{cnX}, and \eqref{Ymat}, respectively, we get
		$$tAYA=\mat{tA_1Y_1A_1}{tA_1Y_2A_2}{tA_2Y_3A_1}{tA_2Y_4A_2},$$
		while for any $s\in\{1,\ldots,n-1\}$, the matrix $A^sYA^{n-s}$ is expressed as
		$$A^sYA^{n-s}=\mat{A_1^sY_1A_1^{n-s}}{A_1^sY_2A_2^{n-s}}{A_2^sY_3A_1^{n-s}}{A_2^sY_4A_2^{n-s}}.$$ 
		Equating the two gives
		\begin{equation}
			\label{tYmat}
			\mat{tA_1Y_1A_1}{tA_1Y_2A_2}{tA_2Y_3A_1}{tA_2Y_4A_2}=\sum_{s=1}^{n-1}\mat{A_1^sY_1A_1^{n-s}}{A_1^sY_2A_2^{n-s}}{A_2^sY_3A_1^{n-s}}{A_2^sY_4A_2^{n-s}}.\end{equation}
		We observe each position $(i,j)$ in the above matrix equation, for $i,j\in\{1,2\}$. By equating the positions $(1,1)$ in \eqref{tYmat} we see that $$tA_1Y_1A_1=\sum\limits_{s=1}^{n-1}A_1^sY_1A_1^{n-s}.$$
		Even without \eqref{A1Y1=Y1A1} we still get \eqref{A1Y1A1}: because $n>2$ we have that $\max\{s,n-s: 1\leq s\leq n-1\}>1$, therefore $A_1^sY_1A_1^{n-s}=0$ for every $s\in\{1,\ldots,n-1\}$, thus \eqref{A1Y1A1} holds. Moreover, since the matrix $A_2$ also has the property \eqref{sqcA}, while $A_1^\ell=0$ for any $\ell>1$, we have that
		\begin{align*} 
			tA_1Y_2A_2 &=A_1YA_2^{n-1}+0+0+\ldots+0=\sum_{s=1}^{n-1}A_1^sY_2A_2^{n-s},\\
			tA_2Y_3A_1 &=0+0+\ldots+0+A_2^{n-1}Y_3A_1=\sum_{s=1}^{n-1}A_2^{s}Y_3A_1^{n-s},
		\end{align*} 
		ergo the positions $(2,1)$ and $(1,2)$ always satisfy the respective equalities in \eqref{tYmat}. Finally, from the position $(2,2)$ we see that the condition \eqref{tY} holds for $Y_4$ with respect to $A_2$: $$tA_2Y_4A_2=\sum_{s=1}^{n-1}A_2^{s}Y_4A_2^{n-s}.
		$$
		The matrix $A_2$ is invertible and it immediately follows that \eqref{tY2} holds. Then, utilizing that $A_2^{n-2}=tI$ and $A_2^{n-2-s}=tA_2^{-s}$ one gets the equality \eqref{0tY2}.
	\end{proof}

	The previous analysis shows that solving Problem ($\mathcal{P}_1$) reduces to solving the system
	\begin{equation}
		\label{XYsystem}
		\left(\mathcal{P}_1'\right):\begin{cases}
			A_1 X_1 A_1 = X_1 A_1 X_1 + \widehat{Y_2} A_2\widehat{Y_3}\\
			X_1 A_1 \widehat{Y_2}+\widehat{Y_2} A_2 X_4= A_1 \widehat{Y_2} A_2\\
			\widehat{Y_3} A_1 X_1+X_4 A_2\widehat{Y_3} = A_2\widehat{Y_3} A_1\\
			A_2 X_4 A_2 =  X_4 A_2 X_4+\widehat{Y_3} A_1 \widehat{Y_2},
		\end{cases}
	\end{equation}
	where the matrices $Y_1$ and $Y_4$ satisfy the equations \eqref{A1Y1=Y1A1}--\eqref{0tY2}. Given that $\widehat{Y_2}$ and $\widehat{Y_3}$ are uniquely determined by the matrices $Y_2$ and $Y_3$, we focus on finding the remaining matrices $X_1$ and $X_4$. In order to do so, we are going to break the system \eqref{XYsystem} into the following subproblems: \\
	
	\noindent $1.$ \emph{Inspecting the consistency of the coupled equations:}
	\begin{equation}
		\label{secthird}
		\begin{cases}X_1 A_1 \widehat{Y_2}+\widehat{Y_2} A_2 X_4  = A_1 \widehat{Y_2} A_2,\\
			\widehat{Y_3} A_1 X_1+
			X_4 A_2\widehat{Y_3}= A_2\widehat{Y_3} A_1.
		\end{cases}
	\end{equation}
	This system is linear with respect to its variables $X_1$ and $X_4$. In Section \ref{SYSTEM} below we solve it completely.\\
	
	\noindent $2.$ \emph{Solving the inhomogeneous Yang-Baxter equations independently from one another:} 
	$$\begin{cases} A_1 X_1 A_1 = X_1 A_1 X_1 + \widehat{Y_2} A_2\widehat{Y_3}\\
		A_2 X_4 A_2 =  X_4 A_2 X_4+\widehat{Y_3} A_1 \widehat{Y_2}.
	\end{cases}$$
	Inhomogeneous Yang-Baxter equations $AXA=XAX+B$ are completely new, and to the best of our nowledge, have not been analyzed nor solved anywhere. In Section \ref{first} below we solve this equation where the coefficient matrix $A$ satisfies $A^2=0$. In Section \ref{fourth} below we solve the inhomogeneous Yang-Baxter equation with invertible coefficient matrix $A$. In each of these cases we obtain necessary and sufficeint conditions for the existence of the respective solutions, and we provide techniques for calculating exact solutions. \\
	
	\noindent $3.$ \emph{Finding the matrices $X_1$ and $X_4$ which solve \emph{both} previous points.}
	
	Once the previous two points are solved, and all the matrices $X_1$ and $X_4$ are accounted for, one takes those matrices which satisfy both of these conditions. That way, the consistency of \eqref{XYsystem} is achieved.

	\subsection{Consistency of System \eqref{secthird}}\label{SYSTEM}
	The consistency of \eqref{secthird} is provided by Theorem \ref{coupled} below:
	\begin{theorem}\label{coupled}
		Consistency of the system \eqref{secthird}
		depends on the matrix
		\begin{equation}\label{matrix_A}
			\A=\left[\begin{array}{cc}
				(A_1 \widehat{Y_2})^T\otimes I_{N_1} & I_{N_2}\otimes \widehat{Y_2} A_2\\
				I_{N_1}\otimes \widehat{Y_3} A_1 & (A_2\widehat{Y_3})^T\otimes I_{N_2}
			\end{array}\right].
		\end{equation}
		The system \eqref{secthird} is consistent if and only if there exists an inner inverse for $\mathcal{A}$ such that
		$$\A\A^{-}\left[\begin{array}{c}
			\operatorname{vec}(A_1 \widehat{Y_2} A_2)\\
			\operatorname{vec}(A_2\widehat{Y_3} A_1)
		\end{array}\right]=\left[\begin{array}{c}
			\operatorname{vec}(A_1 \widehat{Y_2} A_2)\\
			\operatorname{vec}(A_2\widehat{Y_3} A_1)
		\end{array}\right],$$
		and in that case the solutions are given via
		$$\left[\begin{array}{c}
			\operatorname{vec}(X_1)\\
			\operatorname{vec}(X_4)
		\end{array}\right]=\A^{-}\left[\begin{array}{c}
			\operatorname{vec}(A_1 \widehat{Y_2} A_2)\\
			\operatorname{vec}(A_2\widehat{Y_3} A_1)
		\end{array}\right]+(I-\A^{-}\A)\left[\begin{array}{c}
			\operatorname{vec}(Z_1)\\
			\operatorname{vec}(Z_4)
		\end{array}\right],$$
		for any $Z_1,Z_4$ of appropriate dimensions.
	\end{theorem}
	
	\begin{proof}
		By employing the Kronecker product and the vectorization operator, the system \eqref{secthird}, 
		is rewritten as
		\begin{align*}
			\left(\left(A_1 \widehat{Y_2}\right)^T\otimes I_{N_1}\right)\operatorname{vec}(X_1)+\left(I_{N_2}\otimes \widehat{Y_2} A_2\right)\operatorname{vec}(X_4)&=\operatorname{vec}(A_1 \widehat{Y_2} A_2),\\
			\left(I_{N_1}\otimes \widehat{Y_3} A_1\right)\operatorname{vec}(X_1)+\left(\left(A_2\widehat{Y_3}\right)^T\otimes I_{N_2}\right)\operatorname{vec}(X_4)&=\operatorname{vec}(A_2\widehat{Y_3} A_1).
		\end{align*}
		By introducing $\mathcal{A}$ as in \eqref{matrix_A}, and denoting by 
		$$\mathbf{x}=\left[\begin{array}{c}
			\operatorname{vec}(X_1)\\
			\operatorname{vec}(X_4)
		\end{array}\right],\;\mathbf{b}=\left[\begin{array}{c}
			\operatorname{vec}(A_1 \widehat{Y_2} A_2)\\
			\operatorname{vec}(A_2\widehat{Y_3} A_1)
		\end{array}\right],
		$$
		the system \eqref{secthird} reads
		$$\A \mathbf{x}=\mathbf{b}.$$
		From  Theorem \ref{ThmPenrose} it follows that the latter is solvable for $\mathbf{x}$ if and only if $\mathbf{b}\in\R(\mathcal{A})$, that is, if and only if there exists an $\mathcal{A}^-$ such that $\A\A^-\mathbf{b}=\mathbf{b}$.  In that case its general solution is
		$$\mathbf{x}=\A^-\mathbf{b}+(I-\A^-\A)\mathbf{z},$$
		for any vector $\mathbf{z}\in\mathbb{C}^{N_1^2+N_2^2}$. Since the vectorization operator is injective, for every such $\mathbf{z}$ there exist unique matrices $Z_1\in\CC^{N_1\times N_1}$ and $Z_4\in\CC^{N_2\times N_2}$ which determine the vector $\mathbf{z}$ uniquely. Since $\mathbf{z}$ is arbitrary, then so are the matrices $Z_1$ and $Z_4$.
		
	\end{proof}

	\section{The Equation $A_1X_1A_1=X_1A_1X_1+\widehat{Y_2}A_2\widehat{Y_3}$}\label{first}
	In this section we focus on the first equation from \eqref{XYsystem}, that is
	\begin{equation}
		\label{X123}
		A_1 X_1 A_1 - X_1 A_1 X_1 = \widehat{Y_2} A_2\widehat{Y_3}.
	\end{equation}
	This equation is an example of the so-called inhomogeneous Yang-Baxter-like matrix equations. We emphasize that, to the best of our knowledge, they have not yet been investigated in the literature. At this point we recall the first equation from \eqref{Sylsystem}, i.e.,
	\begin{equation}
		\label{Sylnilp2}
		A_1X_1-X_1A_1=Y_1.
	\end{equation}
	Since the latter is a linear equation which is much more simple to solve than \eqref{X123}, we proceed to characterize the solutions to \eqref{Sylnilp2}, and afterwards we inspect which of them solve \eqref{X123} as well.

	\begin{theorem}\label{prvaSyl}
		Let $A_1=\operatorname{diag}(J_2(0),...,J_2(0))=I_k\otimes J_2(0)$ for some $k\geq 1$.
		Then the singular Sylvester equation $A_1X_1-X_1A_1=Y_1$ is consistent if and only if
		\begin{equation}\label{ACCA}
			A_1Y_1+Y_1A_1=0,
		\end{equation}
		and in this case a particular solution is
		\begin{equation}
			X_0 =A_1^\dag Y_1,
		\end{equation}
		where $A_1^\dagger$ is the Moore-Penrose inverse of $A_1$.\end{theorem}

	\begin{proof}
		
		$(\Rightarrow):$ Because of $A_1^2=0$, we immediately have:
		\begin{align*}
			A_1Y_1+Y_1A_1 &=A_1(A_1X_1-X_1A_1)+(A_1X_1-X_1A_1)A_1\\
			& =-A_1X_1A_1+A_1X_1A_1=0.
		\end{align*}
		
		$(\Leftarrow):$ Let the consistency condition $A_1Y_1+Y_1A_1=0$ holds. We will prove that then $X_0=A_1^\dag Y_1$ is a particular solution to the Sylvester equation $A_1X_1-X_1A_1=Y_1$. We have:
		\begin{align*}
			A_1X_0-X_0A_1&=A_1A_1^\dag Y_1-(A_1^\dag Y_1)A_1=A_1A_1^\dag Y_1-A_1^\dag (-A_1Y_1)\\
			&=(A_1A_1^\dag+A_1^\dag A_1)Y_1=Y_1.
		\end{align*}
		To complete the proof, we will show that $A_1A_1^\dag +A_1^\dag A_1=I$ by using Proposition \ref{Kron}. Recall that $A_1^\dag=(I_k\oplus J_2(0))^\dag=I_k\oplus J_2(0)^T$, while the mixed-product property implies
		\begin{align*}
			A_1A_1^\dag&=(I_k\oplus J_2(0))(I_k\oplus J_2(0)^T)=I_k\oplus (J_2(0)J_2(0)^T)=I_k\oplus \left[\begin{array}{cc}
				1 & 0\\
				0 & 0
			\end{array}\right],\\
			A_1^\dag A_1&=(I_k\oplus J_2(0)^T)(I_k\oplus J_2(0))=I_k\oplus (J_2(0)^TJ_2(0))=I_k\oplus \left[\begin{array}{cc}
				0 & 0\\
				0 & 1
			\end{array}\right].
		\end{align*}
		Finally, the linearity gives
		$$A_1A_1^\dag+A_1^\dag A_1=I_k\oplus \left[\begin{array}{cc}
			1 & 0\\
			0 & 0
		\end{array}\right]+I_k\oplus \left[\begin{array}{cc}
			0 & 0\\
			0 & 1
		\end{array}\right]=I_k\otimes I_2=I_{2k},$$
		completing the proof. 
	\end{proof}
	Notice that \eqref{ACCA} coincides with the condition \eqref{A1Y1=Y1A1} from Lemma \ref{necesY1Y4}. Therefore this property is not only necessary for $Y_1$, but is also sufficient.
	
	When a matrix $A_1$ contains at least one $J_1(0)$ block, the following theorem is convenient.
	
	\begin{theorem}\label{Th3.2.}
		Let $A_1=\left[\begin{array}{cc}
			A_{11} & 0\\
			0 & 0
		\end{array}\right]$ where $A_{11}=\operatorname{diag}(J_2(0),...,J_2(0))=I_k\otimes J_2(0)$ for some $k\geq 1$.
		Then the singular Sylvester equation $A_1X_1-X_1A_1=Y_1$ is consistent if and only if both
		\begin{equation}\label{ACCA1}
			A_1Y_1+Y_1A_1=0
		\end{equation}
		and
		\begin{equation}\label{ACCA2}
			(I-A_1A_1^\dag)Y_1(I-A_1^\dag A_1)=0
		\end{equation}
		are satisfied, and in this case a particular solution is
		\begin{equation}\label{X0sol}
			X_0 =A_1^\dag Y_1-JY_1A_1^\dag,\;J=\left[\begin{array}{cc}
				0 & 0\\
				0 & I
			\end{array}\right].
		\end{equation}
	\end{theorem}
	
	\begin{proof}
		$(\Rightarrow):$ Because of $A_1^2=0$ and $A_1A_1^\dag A_1=A_1$, we immediately have:
		\begin{align*}
			A_1Y_1	+YA_1 &= A_1(A_1X_1-X_1A_1)+(A_1X_1-X_1A_1)A_1\\
			& =-A_1X_1A_1+A_1X_1A_1=0,
		\end{align*}
		as well as
		\begin{align*}
			&(I-A_1A_1^\dag)Y_1(I-A_1^\dag A_1)=(I-A_1A_1^\dag)(A_1X_1-X_1A_1)(I-A_1^\dag A_1)\\
			&\phantom{=}=(I-A_1A_1^\dag) A_1X_1(I-A_1^\dag A_1)-(I-A_1A_1^\dag)X_1A_1(I-A_1^\dag A_1)=0.
		\end{align*}
		
		$(\Leftarrow):$ Let the base vectors be arranged such that 
		$$
		A_1=\left[\begin{array}{cc}
			A_{11} & 0\\
			0 & 0
		\end{array}\right],\;A_{11}=I_k\otimes J_2(0),\;k\geq 1,$$
		holds. Respectivedy, decompose $X_1$ and $Y_1$ as
		$$X_1=\left[\begin{array}{cc}
			X_{11} & X_{12}\\
			X_{21} & X_{22}
		\end{array}\right],\qquad
		Y_1=\left[\begin{array}{cc}
			Y_{11} & Y_{12}\\
			Y_{21} & Y_{22}
		\end{array}\right].
		$$
		Then the conditions \eqref{ACCA1}-\eqref{ACCA2} read:
		\begin{align*}
			0&=A_1Y_1+YA_1=\left[\begin{array}{cc}
				A_{11}Y_1+Y_1A_{11} & A_{11}Y_{12}\\
				Y_{21}A_{11} & 0
			\end{array}\right],\\
			0&=(I-A_1A_1^\dag)Y(I-A_1^\dag A_1)\\
			&=\left[\begin{array}{cc}
				(I-A_{11}A_{11}^\dag)Y_{11}(I-A_{11}^\dag A_{11}) & (I-A_{11}A_{11}^\dag)Y_{12}\\
				Y_{21}(I-A_{11}^\dag A_{11}) & Y_{22}
			\end{array}\right].
		\end{align*}
		Hence, we assume that:
		\begin{align*}
			A_{11}Y_{11}+Y_{11}A_{11}&=0,\\
			A_{11}Y_{12}&=0,\\
			Y_{21}A_{11}&=0,\\
			(I-A_{11}A_{11}^\dag)Y_{11}(I-A_{11}^\dag A_{11})&=0,\\
			(I-A_{11}A_{11}^\dag)Y_{12}&=0,\\
			Y_{21}(I-A_{11}^\dag A_{11})&=0,\\
			Y_{22}&=0,
		\end{align*}
		and now we want to find a particular solution to the equation $A_1X_1-X_1A_1=Y_1$, which in terms of the previous calculations reads	$$\left[\begin{array}{cc}
			Y_{11} & Y_{12}\\
			Y_{21} & Y_{22}
		\end{array}\right]=Y_1=A_1X_1-X_1A_1=\left[\begin{array}{cc}
			A_{11}X_{11}-X_{11}A_{11} & A_{11}X_{12}\\
			-X_{21}A_{11} & 0
		\end{array}\right],$$
		that is
		$$
		A_{11}Y_{11}-Y_{11}A_{11}=Y_{11},\;A_{11}X_{12}=Y_{12},\;-X_{21}A_{11}=Y_{21},\;0 =Y_{22}.
		$$
		
		By Theorem \ref{prvaSyl}, $A_{11}X_{11}-X_{11}A_{11}=Y_{11}$ is consistent iff $A_{11}Y_{11}+Y_{11}A_{11}=0$, and then its particular solution is $X_{11}=A_{11}^\dag Y_{11}$. In other words, $A_{11}A_{11}^\dag Y_{11}-A_{11}^\dag Y_{11}A_{11}=Y_{11}$.
		
		Comparing $Y_{12}=A_{11}A_{11}^\dag Y_{12}$ and $A_{11}X_{12}=Y_{12}$, we may take $X_{12}=A_{11}^\dag Y_{12}$. Similarly, comparing $Y_{21}=Y_{21}A_{11}^\dag A_{11}$ and $-X_{21}A_{11}=Y_{21}$, we may take $X_{21}=-Y_{21}A_{11}^\dag$. Hence, we will choose
		$$X_0=\left[\begin{array}{cc}
			X_{11} & X_{12}\\
			X_{21} & X_{22}
		\end{array}\right]=\left[\begin{array}{cc}
			A_{11}^\dag Y_{11} & A_{11}^\dag Y_{12}\\
			-Y_{21}A_{11}^\dag & 0
		\end{array}\right].$$
		
		Now
		\begin{align*}
			A_1X_0-X_0A_1&=\left[\begin{array}{cc}
				A_{11} & 0\\
				0 & 0
			\end{array}\right]\left[\begin{array}{cc}
				A_{11}^\dag Y_{11} & A_{11}^\dag Y_{12}\\
				-Y_{21}A_{11}^\dag & 0
			\end{array}\right]\\
			&\phantom{==}-\left[\begin{array}{cc}
				A_{11}^\dag Y_{11} & A_{11}^\dag Y_{12}\\
				-Y_{21}A_{11}^\dag & 0
			\end{array}\right]\left[\begin{array}{cc}
				A_{11} & 0\\
				0 & 0
			\end{array}\right]\\
			&=\left[\begin{array}{cc}
				A_{11}A_{11}^\dag Y_{11} & A_{11}A_{11}^\dag Y_{12}\\
				0 & 0
			\end{array}\right]-\left[\begin{array}{cc}
				A_{11}^\dag Y_{11}A_{11} & 0\\
				-Y_{21}A_{11}^\dag A_{11} & 0
			\end{array}\right]\\
			&=\left[\begin{array}{cc}
				A_{11}A_{11}^\dag Y_{11}-A_{11}^\dag Y_{11}A_{11} & A_{11}A_{11}^\dag Y_{12}\\
				Y_{21}A_{11}^\dag A_{11} & 0
			\end{array}\right]\\
			&=\left[\begin{array}{cc}
				Y_{11} & Y_{12}\\
				Y_{21} & Y_{22}
			\end{array}\right]=Y_1,
		\end{align*}
		ending the proof.
	\end{proof}

	\begin{remark}
		The consistency condition \eqref{ACCA2} regarding the Moore-Penrose inverses can be rewritten as $Y_{22}=0$, as well as $Y_1(\N(A_1))\subset \R(A_1)$ and $(\R(A_1^\dag))^\bot \subset \N(Y_1)$. In particular, for such $A_1$ it is true $(\R(A_1^\dag))^\bot=\N(A_1)$.
	\end{remark}

	\begin{example}[Nilpotent Sylvester equation]\label{Ex1}
		Let $A_1=J_2(0)\oplus 0=\left[\begin{array}{cc|c}
			0 & 1 & 0\\
			0 & 0 & 0\\
			\hline
			0 & 0 & 0
		\end{array}\right]$ and suppose that $Y_1=[y_{ij}]_{3\times 3}$ and $X_1=[x_{ij}]_{3\times 3}$. The Sylvester equation $A_1X_1-X_1A_1=Y_1$ is equivalent to
		$$\left[\begin{array}{cc|c}
			x_{21} & x_{22}-x_{11} & x_{23}\\
			0 & -x_{21} & 0\\
			\hline
			0 & -x_{31} & 0
		\end{array}\right]=\left[\begin{array}{cc|c}
			y_{11} & y_{12} & y_{13}\\
			y_{21} & y_{22} & y_{23}\\
			\hline
			y_{31} & y_{32} & y_{33}
		\end{array}\right],$$
		which implies the most general form for matrix $Y_1$ establishing consistency as well as the general solution $X_1$:
		$$Y_1=\left[\begin{array}{cc|c}
			y_{11} & y_{12} & y_{13}\\
			0 & -y_{11} & 0\\
			\hline
			0 & y_{32} & 0
		\end{array}\right],\;X_1=\left[\begin{array}{cc|c}
			x_{11} & x_{12} & x_{13}\\
			y_{11} & y_{12}+x_{11} & y_{13}\\
			\hline
			-y_{32} & x_{32} & x_{33}
		\end{array}\right].$$
		Immediate checking shows that $A_1Y_1+Y_1A_1=0$ and $(I-A_1A_1^\dag)Y_1(I-A_1^\dag A_1)=0$, showing that both conditions \eqref{ACCA1}--\eqref{ACCA2} are necessary. Also, the formula \eqref{X0sol} yields
		$$X_0=A_1^\dag Y_1-JY_1A_1^\dag=\left[\begin{array}{cc|c}
			0 & 0 & 0\\
			0 & y_{12} & 0\\
			\hline
			-y_{32} & 0 & 0
		\end{array}\right],$$
		and the immediate checking shows that it is indeed a particular solution of the Sylvester equation.\hfill$\clubsuit$
	\end{example}
	
	\subsection{Nilpotent Yang-Baxter-like Matrix Equation}
	Below we characterize consistency of the equation \eqref{X123}. The condition $A_1^2=0$ is equivalent to the following inclusion: $\R(A_1)\subseteq \N(A_1)$. Hence, there is some subspace $W_1$ such that $\R(A_1)\oplus W_1=\N(A_1)$. There is also some subspace $W_2$ such that $\N(A_1)\oplus W_2=\mathbb{C}^{N_1}$. Since:
	$$A_1(\R(A_1))=0,\;A_1(W_1)=0,\;A_1(W_2)\subseteq \R(A_1),$$
	we may choose the basis of the space such that the block-matrix form of $A_1$ is:
	$$A_1\sim\left[\begin{array}{ccc}
		0 & I_r & 0\\
		0 & 0 & 0\\
		0 & 0 & 0
	\end{array}\right]:\left[\begin{array}{c}
		\R(A_1)\\
		W_2\\
		W_1
	\end{array}\right]\rightarrow \left[\begin{array}{c}
		\R(A_1)\\
		W_2\\
		W_1
	\end{array}\right],$$
	where $r=r(A_1)$.

	\begin{theorem}[Solvability of the inhomogeneous YBME when $A_1^2=0$]\label{inhomYBMEsol}
		Let
		$$
		A_1=S^{-1}\begin{bmatrix}
			0 & I_r & 0\\
			0 & 0 & 0\\
			0 & 0 & 0
		\end{bmatrix}S\in\mathbb{C}^{N_1\times N_1},
		$$
		where $S$ is invertible and $r=\operatorname{rank}(A_1)$.  
		Let $B_1,X_1\in\mathbb{C}^{N_1\times N_1}$ be conformally partitioned as
		$$
		B_1=S^{-1}\begin{bmatrix}
			B_{11} & B_{12} & B_{13}\\
			B_{21} & B_{22} & B_{23}\\
			B_{31} & B_{32} & B_{33}
		\end{bmatrix}S,\;
		X_1=S^{-1}\begin{bmatrix}
			X_{11} & X_{12} & X_{13}\\
			X_{21} & X_{22} & X_{23}\\
			X_{31} & X_{32} & X_{33}
		\end{bmatrix}S.
		$$
		Define
		$$
		M := [\,X_{21}\;\;X_{22}\;\;X_{23}\,],\;
		R_2 := [\,B_{21}\;\;B_{22}\;\;B_{23}\,],\;
		R_3 := [\,B_{31}\;\;B_{32}\;\;B_{33}\,],
		$$
		and
		$$
		R_1(X_{21}) := [\,B_{11}\;\;X_{21}-B_{12}\;\;B_{13}\,].
		$$
		
		Then the inhomogeneous Yang-Baxter-like matrix equation
		$$
		A_1X_1A_1 - X_1A_1X_1 = B_1
		$$
		is consistent if and only if there exists a matrix $X_{21}$ such that
		\begin{equation}\label{cond1}
			X_{21}^2 = -B_{21},
		\end{equation}
		and, for this fixed choice of $X_{21}$, there exist matrices $X_{22},X_{23}$ satisfying
		\begin{equation}\label{cond2}
			X_{21}X_{22}=-B_{22},\;
			X_{21}X_{23}=-B_{23},
		\end{equation}
		and the range conditions
		\begin{equation}\label{cond3}
			R_3 = R_3 M^- M,\qquad
			R_1(X_{21}) = R_1(X_{21}) M^- M
		\end{equation}
		hold for some (equivalently, any) inner inverse $M^-$ of $M$.
		
		For each fixed admissible $X_{21}$, all solutions $X_1$ are obtained as follows:
		\begin{itemize}
			\item[(i)] Choose any $X_{21}$ satisfying \eqref{cond1}.
			\item[(ii)] For any inner inverse $X_{21}^-$,
			\[
			\begin{aligned}
				X_{22}&=-X_{21}^-B_{22}+(I-X_{21}^-X_{21})Z_{22},\\
				X_{23}&=-X_{21}^-B_{23}+(I-X_{21}^-X_{21})Z_{23},
			\end{aligned}
			\]
			with arbitrary parameters $Z_{22},Z_{23}$, provided
			\[
			X_{21}X_{21}^-B_{22}=B_{22},\qquad
			X_{21}X_{21}^-B_{23}=B_{23}.
			\]
			\item[(iii)] For any inner inverse $M^-$,
			\[
			\begin{aligned}
				X_{31}&=-R_3M^-+Z_{31}(I-MM^-),\\
				X_{11}&=R_1(X_{21})M^-+Z_{11}(I-MM^-),
			\end{aligned}
			\]
			with arbitrary parameters $Z_{31},Z_{11}$.
			\item[(iv)] The blocks $X_{12},X_{13},X_{32},X_{33}$ are arbitrary.
		\end{itemize}
	\end{theorem}
	
	\begin{proof}
		Using the similarity invariance of the equation $A_1X_1A_1-X_1A_1X_1=B_1$, we equivalently solve
		$$
		\widetilde{A_1} \widetilde{X_1} \widetilde{A_1} - \widetilde{X_1} \widetilde{A_1} \widetilde{X_1} = \widetilde{B_1},
		$$
		where
		$$
		\widetilde{A_1}=\begin{bmatrix}
			0 & I_r & 0\\
			0 & 0 & 0\\
			0 & 0 & 0
		\end{bmatrix},\quad
		\widetilde{X_1}=\begin{bmatrix}
			X_{11} & X_{12} & X_{13}\\
			X_{21} & X_{22} & X_{23}\\
			X_{31} & X_{32} & X_{33}
		\end{bmatrix},\quad
		\widetilde{B_1}=\begin{bmatrix}
			B_{11} & B_{12} & B_{13}\\
			B_{21} & B_{22} & B_{23}\\
			B_{31} & B_{32} & B_{33}
		\end{bmatrix}.
		$$
		
		A direct block computation shows that this equation is equivalent to the following system:
		\begin{align}
			-X_{11}X_{21} &= B_{11}, \label{eq11}\\
			X_{21}-X_{11}X_{22} &= B_{12}, \label{eq12}\\
			-X_{11}X_{23} &= B_{13}, \label{eq13}\\
			-X_{21}^2 &= B_{21}, \label{eq21}\\
			-X_{21}X_{22} &= B_{22}, \label{eq22}\\
			-X_{21}X_{23} &= B_{23}, \label{eq23}\\
			-X_{31}X_{21} &= B_{31}, \label{eq31}\\
			-X_{31}X_{22} &= B_{32}, \label{eq32}\\
			-X_{31}X_{23} &= B_{33}. \label{eq33}
		\end{align}
		The blocks $X_{12},X_{13},X_{32},X_{33}$ do not appear and are therefore arbitrary.
		
		\medskip
		\noindent\emph{Necessity.}
		Equation \eqref{eq21} implies the existence of a matrix $X_{21}$ satisfying
		$$
		X_{21}^2=-B_{21}.
		$$
		Fix such a choice of $X_{21}$.  
		Equations \eqref{eq22} and \eqref{eq23} are solvable if and only if
		$$
		X_{21}X_{21}^-B_{22}=B_{22},\qquad
		X_{21}X_{21}^-B_{23}=B_{23},
		$$
		for some (equivalently, any) inner inverse $X_{21}^-$, in which case their general solutions are
		$$
		\begin{aligned}
			X_{22}&=-X_{21}^-B_{22}+(I-X_{21}^-X_{21})Z_{22},\\
			X_{23}&=-X_{21}^-B_{23}+(I-X_{21}^-X_{21})Z_{23}.
		\end{aligned}
		$$
		
		Let $M=[\,X_{21}\;\;X_{22}\;\;X_{23}\,]$.  
		Then equations \eqref{eq31}–\eqref{eq33} are equivalent to
		$$
		X_{31}M=-R_3,
		$$
		which is solvable if and only if $R_3=R_3M^-M$ for some inner inverse $M^-$.  
		Similarly, equations \eqref{eq11}–\eqref{eq13} are equivalent to
		$$
		X_{11}M=R_1(X_{21}),
		$$
		which is solvable if and only if $R_1(X_{21})=R_1(X_{21})M^-M$.
		
		Thus all stated conditions are necessary.
		
		\medskip
		\noindent\emph{Sufficiency.}
		Assume that there exists a matrix $X_{21}$ satisfying $X_{21}^2=-B_{21}$ and that the range conditions hold.  
		Define $X_{22},X_{23}$ as above, ensuring \eqref{eq22} and \eqref{eq23}.  
		Next define
		$$
		\begin{aligned}
			X_{31}&=-R_3M^-+Z_{31}(I-MM^-),\\
			X_{11}&=R_1(X_{21})M^-+Z_{11}(I-MM^-),
		\end{aligned}
		$$
		which satisfy equations \eqref{eq11}–\eqref{eq13} and \eqref{eq31}–\eqref{eq33}.  
		All nine block equations are therefore satisfied, and hence
		$$
		\widetilde{A_1} \widetilde{X_1} \widetilde{A_1} - \widetilde{X_1} \widetilde{A_1} \widetilde{X_1} = \widetilde{B_1}.
		$$
		Applying the inverse similarity transformation completes the proof.
	\end{proof}
	
	\begin{remark}[On the square-root condition]
		The condition \eqref{cond1} requires the existence of a matrix square root of $-B_{21}$. 
		Such a square root need not exist in general and, when it does, is typically not unique. We emphasize that there are square complex matrices that don't have a square root, e.g. $J_2(0)$. For more details see \cite{Higham}.
		Each admissible choice of $X_{21}$ satisfying $X_{21}^2=-B_{21}$ generates a distinct affine family of solutions of the inhomogeneous YBME.  
		In particular, the solution set may be disconnected and its structure depends on the Jordan canonical form of $B_{21}$.
	\end{remark}
	
	It is a good time to state the particular case when the previous theorem deals with homogeneous YBME, that is, when $B_1=0$.
	
	\begin{theorem}[Solvability of the homogeneous YBME when $A_1^2=0$]\label{YBME_nil_2}
		Let
		$$
		A_1=S^{-1}\begin{bmatrix}
			0 & I_r & 0\\
			0 & 0 & 0\\
			0 & 0 & 0
		\end{bmatrix}S\in\mathbb{C}^{N_1\times N_1},
		$$
		where $S$ is invertible and $r=\operatorname{rank}(A)$.  
		Let $X_1\in\mathbb{C}^{N_1\times N_1}$ be conformally partitioned as
		$$
		X_1=S^{-1}\begin{bmatrix}
			X_{11} & X_{12} & X_{13}\\
			X_{21} & X_{22} & X_{23}\\
			X_{31} & X_{32} & X_{33}
		\end{bmatrix}S.
		$$
		Define
		$M := [\,X_{21}\;\;X_{22}\;\;X_{23}\,].
		$
		Then the (homogeneous) Yang-Baxter-like matrix equation
		$$
		A_1X_1A_1 = X_1A_1X_1 
		$$
		is consistent.
		
		For each nonzero nilpotent $X_{21}$ with $X_{21}^2=0$, all solutions $X_1$ are obtained as follows:
		\begin{itemize}
			\item[(i)] For any inner inverse $X_{21}^-$,
			$$
			X_{22}=(I-X_{21}^-X_{21})Z_{22},\;
			X_{23}=(I-X_{21}^-X_{21})Z_{23},
			$$
			with arbitrary parameters $Z_{22},Z_{23}$.
			\item[(ii)] For any inner inverse $M^-$,
			$$
			X_{31}=Z_{31}(I-MM^-),\;
			X_{11}=[\,0\;X_{21}\;0\,]M^-+Z_{11}(I-MM^-),
			$$
			with arbitrary parameters $Z_{31},Z_{11}$.
			\item[(iii)] The blocks $X_{12},X_{13},X_{32},X_{33}$ are arbitrary.
		\end{itemize}
	\end{theorem}
	\begin{corollary}
		The solutions to the YBME $A_1X_1A_1=X_1A_1X_1$ with $A_1^2=0$ are given by 
		$$X_1=S^{-1}\left[\begin{array}{ccc}
			Z_{11}(I-MM^-) & Z_{12} & Z_{13}\\
			0 & Z_{22} & Z_{23}\\
			Z_{31}(I-MM^-) & M_{32} & M_{33}
		\end{array}\right]S,$$
		with $M=[\,0\;X_{22}\;X_{23}\,]$ and arbitrary matrices $Z=[Z_{ij}]_{3\times 3}$.
	\end{corollary}
	
	\begin{example}[Homogeneous YBME when $ A_1^2=0$]\label{EX1} Recall Example \ref{Ex1}: let $A_1=J_2(0)\oplus 0_1=\left[\begin{array}{cc|c}
			0 & 1 & 0\\
			0 & 0 & 0\\
			\hline
			0 & 0 & 0
		\end{array}\right],$ and let the commutator matrix $Y_1$ be provided as
		$$Y_1=\left[\begin{array}{cc|c}
			y_{11} & y_{12} & y_{13}\\
			0 & -y_{11} & 0\\
			\hline
			0 & y_{32} & 0
		\end{array}\right].$$ Then the general solution to the corresponding Sylvester equation $A_1X_1-X_1A_1=Y_1$ is given as
		$$X_1=\left[\begin{array}{cc|c}
			x_{11} & x_{12} & x_{13}\\
			y_{11} & y_{12}+x_{11} & y_{13}\\
			\hline
			-y_{32} & x_{32} & x_{33}
		\end{array}\right].$$
		Now, let us filter those matrices $X_1$ that are solutions to the homogeneous YBME $A_1X_1A_1=X_1A_1X_1$. We arrive at
		$$\left[\begin{array}{cc|c}
			0 & y_{11} & 0\\
			0 & 0 & 0\\
			\hline
			0 & 0 & 0
		\end{array}\right]=\left[\begin{array}{cc|c}
			y_{11}x_{11} & x_{11}(y_{12}+x_{11}) & y_{13}x_{11}\\
			y_{11}^2 & y_{11}(y_{12}+x_{11}) & y_{11}y_{13}\\
			\hline
			-y_{11}y_{32} & -y_{32}(y_{12}+x_{11}) & -y_{13}y_{32}
		\end{array}\right],$$
		which after some computation leads to the following four solutions to the YBME (we will update the corresponding commutators $Y_1$ for the sake of completeness):
		\begin{align*}
			1)\;& X_1(1)=\left[\begin{array}{cc|c}
				-y_{12} & x_{12} & x_{13}\\
				0 & 0 & 0\\
				\hline
				-y_{32} & x_{32} & x_{33}
			\end{array}\right],\qquad Y_1(1)=\left[\begin{array}{cc|c}
				0 & y_{12} & 0\\
				0 & 0 & 0\\
				\hline
				0 & y_{32} & 0
			\end{array}\right],\\
			2)\;& X_1(2)=\left[\begin{array}{cc|c}
				0 & x_{12} & x_{13}\\
				0 & y_{12} & y_{13}\\
				\hline
				0 & x_{32} & x_{33}
			\end{array}\right],\qquad Y_1(2)=\left[\begin{array}{cc|c}
				0 & y_{12} & y_{13}\\
				0 & 0 & 0\\
				\hline
				0 & 0 & 0
			\end{array}\right],\\
			3)\;& X_1(3)=\left[\begin{array}{cc|c}
				0 & x_{12} & x_{13}\\
				0 & y_{12} & 0\\
				\hline
				0 & x_{32} & x_{33}
			\end{array}\right],\qquad Y_1(3)=\left[\begin{array}{cc|c}
				0 & y_{12} & 0\\
				0 & 0 & 0\\
				\hline
				0 & 0 & 0
			\end{array}\right],\\
			4)\;& X_1(4)=\left[\begin{array}{cc|c}
				-y_{12} & x_{12} & x_{13}\\
				0 & 0 & 0\\
				\hline
				0 & x_{32} & x_{33}
			\end{array}\right],\qquad Y_1(4)=\left[\begin{array}{cc|c}
				0 & y_{12} & 0\\
				0 & 0 & 0\\
				\hline
				0 & 0 & 0
			\end{array}\right].
		\end{align*}
		We emphasize the the immediate checking shows that the solutions $X_1(2),$ $X_1(3),$ $X_1(4)$ are solutions not only to the Sylvester equations (for $Y_1(2),$ $Y_1(3),$ $Y_1(4)$, respectively), but for the homogeneous YBME, too!\hfill$\clubsuit$
	\end{example}

	\section{The Equation $A_2 X_4 A_2 =  X_4 A_2 X_4+\widehat{Y_3} A_1 \widehat{Y_2}$}\label{fourth}
	In this section we solve the fourth equation from the system \eqref{XYsystem}:
	\begin{equation}
		\label{X432}
		A_2 X_4 A_2 = X_4 A_2 X_4+\widehat{Y_3} A_1 \widehat{Y_2}.
	\end{equation}
	We recall the commutator $Y_4$ for $A_2$ and $X_4$:
	\begin{equation}
		\label{Y2}
		Y_4:=A_2X_4-X_4A_2,
	\end{equation}
	which, in general, can be zero, but if it is not then it must satisfy some natural conditions like \eqref{tY2}--\eqref{0tY2} as before. Since $A_2$ is invertible and satisfies $A_2^n=tA_2^2$, we have the following result:
	\begin{theorem}\label{Th4.1.}
		Let $
		A_2=T\operatorname{diag}(\lambda_1 I_{r_1},\dots,\lambda_s I_{r_s})T^{-1}
		$, where $\lambda_1,\dots,\lambda_s$ are distinct roots of $x^{n-2}=t$ (possibly not all of them) and
		$r_1+\cdots+r_s=N_2$.
		The equation
		$
		A_2X_4-X_4A_2=Y_4
		$
		is consistent if and only if
		\begin{equation}\label{cons_cond_inv}
			(\forall k\in\{1,2,...,s\})\;P_k Y_4P_k = 0,
		\end{equation}
		where $P_k$ is the spectral projector of $A$ onto the $\lambda_k-$eigenspace.
		
		Then all the solutions are given by
		\begin{equation}\label{all_sol_inv}
			X_4=\sum_{\substack{i,j=1\\ i\neq j}}^s \frac{1}{\lambda_i-\lambda_j}P_iY_4P_j+\sum_{i=1}^{s}P_i Z_i P_i,\quad \forall Z_i\in\CC^{N_2\times N_2},i=1,...,s.
		\end{equation}
		Equivalently, in a basis where $A_2=\operatorname{diag}(\lambda_1 I_{r_1},\ldots,\lambda_s I_{r_s})$, the homogeneous part
		is an arbitrary block-diagonal matrix $\operatorname{diag}(U_1,\ldots,U_s)$ with $U_i\in\CC^{r_i\times r_i}$.
	\end{theorem}
	
	\begin{proof}
		We use the standard identities for spectral projectors:
		\[
		P_iP_j=\delta_{ij}P_i,\qquad \sum_{i=1}^s P_i = I,\qquad
		A_2=\sum_{i=1}^s \lambda_i P_i,\qquad A_2P_i=P_iA_2=\lambda_i P_i.
		\]
		$(\Rightarrow):$ Assume there exists $X_4$ such that $A_2X_4-X_4A_2=Y_4$.
		Fix $k\in\{1,\ldots,s\}$ and multiply the equation on the left and right by $P_k$; we obtain
		$$
		P_kY_4P_k=(P_kA_2)X_4P_k - P_kX_4(A_2P_k)
		=(\lambda_k P_k)X_4P_k - P_kX_4(\lambda_k P_k)=0,
		$$
		proving \eqref{cons_cond_inv}.
		
		\noindent$(\Leftarrow):$ Assume \eqref{cons_cond_inv} holds. Decompose $Y_4$ via $\sum_i P_i=I$:
		\[
		Y_4 = \Big(\sum_{i=1}^s P_i\Big)Y_4\Big(\sum_{j=1}^s P_j\Big)
		= \sum_{i,j=1}^s P_iY_4P_j
		= \sum_{i\neq j} P_iY_4P_j \;+\; \sum_{k=1}^s P_kY_4P_k.
		\]
		By \eqref{cons_cond_inv} the diagonal part vanishes, so
		\begin{equation}\label{eq:Y-offdiag}
			Y_4 = \sum_{i\neq j} P_iY_4P_j.
		\end{equation}
		Define
		\[
		X_p := \sum_{\substack{i,j=1\\ i\neq j}}^{s}\frac{1}{\lambda_i-\lambda_j}\,P_iY_4P_j.
		\]
		For each $i\neq j$, we compute
		\[
		(A_2P_i)Y_4P_j - P_iY_4(P_j A_2)
		= (\lambda_i P_i) Y_4P_j - P_iY_4 (\lambda_j P_j)
		= (\lambda_i-\lambda_j)P_iY_4P_j.
		\]
		Therefore,
		\[
		A_2\Big(\frac{1}{\lambda_i-\lambda_j}P_iY_4P_j\Big)
		-\Big(\frac{1}{\lambda_i-\lambda_j}P_iY_4P_j\Big)A_2
		= P_iY_4P_j.
		\]
		Summing over all $i\neq j$ yields
		\[
		A_2X_p - X_pA_2 = \sum_{i\neq j} P_iY_4P_j = Y_4.
		\]
		Hence the equation is consistent and $X_p$ is a particular solution.
		
		Now we address homogeneous solutions and the general form.
		Let $X_4$ be any solution and set $X_h := X_4 - X_p$. Then
		$
		A_2X_h - X_hA_2 = 0,
		$
		i.e., $X_h$ commutes with $A_2$.
		
		We claim that
		\begin{equation}\label{eq:block-characterization}
			[A_2,X]=0 \quad\Longleftrightarrow\quad P_iXP_j=0\ \text{for all}\ i\neq j,
			\quad\Longleftrightarrow\quad
			X=\sum_{i=1}^s P_iXP_i.
		\end{equation}
		Indeed, write $X=\sum_{i,j}P_iXP_j$. For each pair $(i,j)$,
		\[
		(A_2P_i)XP_j - P_iX(P_jA_2)
		= (\lambda_i-\lambda_j)P_iXP_j.
		\]
		If $[A_2,X]=0$, then summing over $(i,j)$ gives
		\[
		0=[A_2,X]=\sum_{i,j}(\lambda_i-\lambda_j)P_iXP_j.
		\]
		Multiplying on the left by $P_i$ and on the right by $P_j$ isolates a single term:
		\[
		0=P_i[A_2,X]P_j=(\lambda_i-\lambda_j)P_iXP_j.
		\]
		Since $\lambda_i\neq\lambda_j$ for $i\neq j$, we obtain $P_iXP_j=0$ whenever $i\neq j$,
		which proves the first implication in \eqref{eq:block-characterization}; the remaining equivalences are immediate.
		
		Consequently, any homogeneous solution has the form
		\[
		X_h=\sum_{i=1}^s P_i X_h P_i.
		\]
		Conversely, for arbitrary matrices $Z_1,\ldots,Z_s\in\mathbb{C}^{N_2\times N_2}$ the matrix
		$X_h:=\sum_{i=1}^s P_iZ_iP_i$ satisfies $[A_2,X_h]=0$, because it contains no off-diagonal
		blocks relative to the spectral decomposition of $A_2$.
		Therefore, every solution is
		\[
		X_4=X_p+X_h
		=
		\sum_{\substack{i,j=1\\ i\neq j}}^{s}\frac{1}{\lambda_i-\lambda_j}\,P_iY_4P_j
		\;+\;
		\sum_{i=1}^{s} P_i Z_i P_i,
		\]
		which is exactly \eqref{all_sol_inv}.
	\end{proof}

	Now we can use the matrices $X_4$ from \eqref{all_sol_inv} to search for the solutions to \eqref{X432}. This can be simplified even further in the following sense:
	
	\begin{lemma}[Necessary condition for $A_2X_4A_2$ and $X_4A_2X_4$]
		Let $A_2$ be an invertible square matrix such that $A_2 ^n=tA_2^2$ holds for some $n>2$ and $t\neq0$. Let $X_4$ and $Y_4$ be provided such that \eqref{Y2} is satisfied. Then the following equalities hold:
		
		\begin{equation}\label{AAXAXA}
			A_2^2(X_4A_2X_4)A_2=A_2Y_4A_2X_4A_2+(A_2X_4A_2)^2.
		\end{equation}
		\begin{equation}\label{AXAXAA}A_2(X_4A_2X_4)A_2^2=(A_2X_4A_2)^2-A_2X_4A_2Y_4A_2.
		\end{equation}
		
	\end{lemma}
	\begin{proof} We proceed to analyze the expression $X_4A_2X_4$. From $tA_2^2=A_2^n$ we have
		\begin{eqnarray}
			\begin{aligned}
				(tA_2X_4A_2)^2 &=t^2A_2X_4A_2^2X_4A_2=tA_2X_4A_2^nX_4A_2\\
				&=tA_2(A_2X_4-Y_4)A_2^{n-1}X_4A_2\\
				&=t^2A_2^2X_4A_2X_4A_2-t^2A_2Y_4A_2X_4A_2 \\
				&\Rightarrow A_2^2(X_4A_2X_4)A_2=A_2Y_4A_2X_4A_2+(A_2X_4A_2)^2.
			\end{aligned}
		\end{eqnarray}
		and
		\begin{eqnarray}
			\begin{aligned}
				(tA_2X_4A_2)^2 &=t^2A_2X_4A_2^2X_4A_2=tA_2X_4A_2^nX_4A_2\\
				&=tA_2X_4A_2^{n-1}(X_4A_2+Y_4)A_2\\
				&=t^2A_2X_4A_2^{n-1}X_4A_2^2+t^2A_2X_4A_2^{n-1}Y_4A_2\\
				&\Rightarrow A_2(X_4A_2X_4)A_2^2=(A_2X_4A_2)^2-A_2X_4A_2Y_4A_2.
			\end{aligned}
		\end{eqnarray}
	\end{proof}
	
	\begin{theorem}[On the existence of $X_4$] \label{SylcondtX4} Let $A_2$ be an invertible square matrix such that $A_2 ^n=tA_2^2$ holds for some $n>2$ and $t\neq0$, and let $Y_4$ be arbitrary. 
		
		If the system \eqref{X432}--\eqref{Y2} is consistent, then every solution $X_4$ to \eqref{X432}--\eqref{Y2} also solves the Sylvester equation
		\begin{equation}\label{XYZSyl}
			Y_4A_2X_4+X_4A_2Y_4=A_2Y_4A_2+\left(\widehat{Y_3} A_1 \widehat{Y_2}\right)A_2-A_2\left(\widehat{Y_3} A_1 \widehat{Y_2}\right).
		\end{equation}
		Conversely, if the equation \eqref{XYZSyl} is unsolvable for $X_4$, then the system \eqref{X432}--\eqref{Y2} is inconsistent. 
	\end{theorem}
	\begin{proof} 
		By rewriting $A_2X_4A_2=X_4A_2X_4+\widehat{Y_3} A_1 \widehat{Y_2}$ we get
		\begin{eqnarray}
			\label{X4A2X4}
			\begin{aligned}
				&A_2X_4-X_4A_2=Y_4\Leftrightarrow A_2(A_2X_4A_2)-(A_2X_4A_2)A_2=A_2Y_4A_2 \\
				&\Leftrightarrow A_2(X_4A_2X_4+\widehat{Y_3} A_1 \widehat{Y_2})-(X_4A_2X_4+\widehat{Y_3} A_1 \widehat{Y_2})A_2=A_2Y_4A_2\\
				&\Leftrightarrow A_2(X_4A_2X_4)-(X_4A_2X_4)A_2=A_2Y_4A_2+\left(\widehat{Y_3} A_1 \widehat{Y_2}\right)A_2-A_2\left(\widehat{Y_3} A_1 \widehat{Y_2}\right)\\
				&\Leftrightarrow A_2^2(X_4A_2X_4)A_2-A_2(X_4A_2X_4)A_2^2=\\
				&\qquad A_2^2Y_4A_2^2+A_2\left(\widehat{Y_3} A_1 \widehat{Y_2}\right)A_2^2-A_2^2\left(\widehat{Y_3} A_1 \widehat{Y_2}\right)A_2.
			\end{aligned}
		\end{eqnarray}
		Combining the equalities \eqref{AAXAXA} and \eqref{AXAXAA} gives
		\begin{equation}
			\label{YZsquared}
			A_2Y_4A_2X_4A_2+A_2X_4A_2Y_4A_2=A_2^2Y_4A_2^2+A_2\left(\widehat{Y_3} A_1 \widehat{Y_2}\right)A_2^2-A_2^2\left(\widehat{Y_3} A_1 \widehat{Y_2}\right)A_2,
		\end{equation}
		that is
		\begin{equation}
			\label{YZ}
			Y_4A_2X_4+X_4A_2Y_4=A_2Y_4A_2+\left(\widehat{Y_3} A_1 \widehat{Y_2}\right)A_2-A_2\left(\widehat{Y_3} A_1 \widehat{Y_2}\right).
		\end{equation}
		The converse statement follows immediately.
	\end{proof}
	
	\begin{corollary}[Homogeneous YBME with invertible square-cyclic $A_2$]\label{invhomYBME} Let $A_2$ be an invertible square matrix such that $A_2 ^n=tA_2^2$ holds for some $n>2$ and $t\neq0$, and let $Y_4$ be arbitrary. If the system 
		\begin{equation}\label{systregA2}\begin{cases}A_2X_4A_2=X_4A_2X_4,\\A_2X_4-X_4A_2=Y_4\end{cases}
		\end{equation} is consistent, then every solution $X_4$ also solves the Sylvester equation
		\begin{equation}\label{regSylA2}
			Y_4A_2X_4+X_4A_2Y_4=A_2Y_4A_2.
		\end{equation}
		Conversely, if the above Sylvester equation \eqref{regSylA2} is unsolvable for $X_4$, then the afore-mentioned system \eqref{systregA2} is inconsistent. 
	\end{corollary}
	
	\begin{corollary} Let $A_2$ be an invertible square matrix such that $A_2 ^n=tA_2^2$ holds for some $n>2$ and $t\neq0$, and let $Y_4$ be arbitrary. 
		
		If $Y_4$ is invertible, and provided in the manner that $\sigma(A_2Y_4)$ does not contain two opposite points, then \eqref{XYZSyl} has a unique solution $X_4$, and the system \eqref{X432}--\eqref{Y2} is either inconsistent, or has only the one solution and that is $X_4$.
	\end{corollary}
	\begin{proof} If $Y_4$ is invertible, then so are the matrices $A_2Y_4$ and $Y_4A_2$. Moreover, by the Gel'fand transform of commutative Banach algebras, it follows that
		$$\sigma(A_2Y_4)\cup\{0\}=\sigma(Y_4A_2)\cup\{0\}$$
		which implies that 
		$$\sigma(A_2Y_4)=\sigma(Y_4A_2).$$
		Furthermore, the assumption that $\sigma(A_2Y_4)$ does not contain two opposite points implies that
		$$\emptyset=\sigma(A_2Y_4)\cap\sigma(-A_2Y_4)\Leftrightarrow \sigma(Y_4A_2)\cap\sigma(-A_2Y_4)=\emptyset,$$
		therefore \eqref{XYZSyl} is a regular Sylvester equation, with a unique solution $X_4$. Then by the previous theorem (Theorem \ref{SylcondtX4}), the system \eqref{X432}--\eqref{Y2} is either inconsistent, or has only a unique solution and that is $X_4$.
	\end{proof}
	
	On the other hand,  if $Y_4$ is chosen in the manner that 
	$$\sigma(Y_4A_2)\cap\sigma(-A_2Y_4)\neq\emptyset,$$
	then \eqref{XYZSyl} (and \eqref{regSylA2}) could be inconsistent, or it could have infinitely many solutions. We refer to the results obtained e.g. in \cite{DIN}, \cite{NCDBDD2}, \cite{BDD}, \cite{BDND1}.

	\section{Some special cases}
	As previously mentioned, our results hold even when the entire matrix $Y$, or some of its blocks are zeros. In those instances, the calculations conducted above simplify drastically. Below we enlist some of them.
	
	\subsection{Block-diagonal matrix $Y$}
	
	If the matrix $Y$ is block-diagonal in the representation \eqref{Ymat}, then $Y_2=0$ and $Y_3=0$, which implies that $X_2=0$ and $X_3=0$ (zeros in the appropriate operator spaces), thus the matrix representation of $X$ is block-diagonal with respect to \eqref{CN} as well, that is,
	\begin{equation}
		\label{blockdiagX}X=\mat{X_1}{0}{0}{X_4}.
	\end{equation} 
	Also note that the converse statement holds, i.e., if $X$ is in the form \eqref{blockdiagX} then $Y=\mat{Y_1}{0}{0}{Y_4}$, by the virtue of \eqref{AXXAYmat}. However, note that this does not imply that the solutions of the form \eqref{blockdiagX}  are automatically commuting ones, although all commuting solutions are obtained within this class (that is, when $Y=0$). Accordingly, the system  \eqref{XYsystem}  reduces to
	\begin{equation}\label{systemred}
		\begin{cases} A_1X_1A_1=X_1A_1X_1\\A_2 X_4 A_2=X_4A_2X_4.
		\end{cases}
	\end{equation}
	The two equations in \eqref{systemred} are independent from one another. Each of these equations can be solved by the techniques described in this paper: in particular, for the first equation one applies Section \ref{first} Theorem \ref{YBME_nil_2}, while for the second equation one applies Section \ref{fourth} Corollary \ref{invhomYBME}. We point out that \cite{DZJD} also concerns the first equation $A_1X_1A_1=X_1A_1X_1$, however, these results are more oriented towards commuting solutions with only a (not very practical) hint on how to theoretically obtain the non-commuting solutions.

	\subsection{Commuting solutions: $Y=0$}\label{commuting}
	
	Specially, when $Y=0$, then the matrix $X$ retains the block form \eqref{blockdiagX}, however, the system \eqref{systemred} now concerns only the commuting solutions $X_1$ and $X_4$. Note that all the commuting solutions must be singular:
	\begin{theorem} \label{singular}
		If $A_2$ is an invertible matrix, then all nontrivial commuting solutions to $A_2X_4A_2=X_4A_2X_4$ are singular.
	\end{theorem}
	
	\begin{proof}
		It is easy to show that if we assume $X_4$ is a commuting solution, the equation is equivalent to $A_2X_4 = X_4^2 = X_4A_2$.
		If $X_4$ is invertible, it follows that $A_2X_4 = X_4^2$ is equivalent to $A_2 = X_4$. Ergo if $X_4\neq A_2$ then $X_4$ is singular.
	\end{proof}
	
	All commuting solutions are accounted for: for the square-nilpotent matrix $A_1$, all commuting solutions are collected in \cite{Dong2018}. For the invertible matrix $A_2$ (regardless of the square-cyclic condition), all commuting solutions are collected in \cite{NCDBDD2}, while, since $A_2$ can be diagonalized, one can use the paper \cite{DonDin2017} as well.

	\subsection{The case when $Y_2=0$ or $Y_3=0$}
	
	By assuming that $Y_2=0$ or $Y_3=0$ we again have the reduced equations from \eqref{systemred}, but with one more equation:
	\begin{equation}
		\label{Y3=0}
		\begin{cases}
			A_1X_1A_1=X_1A_1X_1\\
			X_1 A_1 \widehat{Y_2}  = A_1 \widehat{Y_2} A_2- \widehat{Y_2} A_2 X_4,\\
			A_2X_4A_2=X_4A_2X_4.
		\end{cases}
	\end{equation} 
	or
	\begin{equation}
		\label{Y2=0}
		\begin{cases}
			A_1X_1A_1=X_1A_1X_1\\
			X_4 A_2\widehat{Y_3} = A_2\widehat{Y_3} A_1-\widehat{Y_3} A_1 X_1,\\
			A_2X_4A_2=X_4A_2X_4.
		\end{cases}
	\end{equation} 
	Similalry as for the block-diagonal case, we solve the system \eqref{systemred} and find all its solutions $X_1$ and $X_4$, and afterwards we utilize Theorem \ref{coupled} to inspect whether the consistency criteria for the obtained $X_4$ and $X_1$ are met.
	
	\subsection{Anti-diagonal matrix $Y$}
	
	We now restrict our attention to the case where $Y$ is anti-diagonal with respect to the decomposition \eqref{CN}, i.e., $Y_1=0$ and $Y_4=0$, while $Y_2$ and $Y_3$ are arbitrary nonzero matrices.
	
	By the definition of $Y$ in \eqref{Ymat}, the vanishing diagonal blocks imply that the solution blocks $X_1$ and $X_4$ must commute with the corresponding blocks of $A$:
	\begin{equation}
		\label{commute_anti}
		A_1 X_1  = X_1 A_1 \quad \text{and} \quad A_2 X_4 = X_4 A_2.
	\end{equation}
	By Theorem \ref{nil_inv_Syl}, we have
	\begin{align}
		\widehat{Y_2} &= -Y_2 A_2^{-1} - A_1 Y_2 A_2^{-2}, \label{Y2trunc}\\
		\widehat{Y_3} &= A_2^{-1} Y_3 + A_2^{-2} Y_3 A_1. \label{Y3trunc}
	\end{align}
	The nilpotency of $A_1$ imposes strict consistency conditions on the input matrices $Y_2$ and $Y_3$. Utilizing $A_1 X_1=X_1 A_1$ and $A_1^2=0$ in the first equation of \eqref{XYsystem}, we obtain $X_1^2 A_1 = -\widehat{Y_2} A_2 \widehat{Y_3}$. Post-multiplying this by $A_1$ annihilates the left-hand side, yielding the constraint $\widehat{Y_2} A_2 \widehat{Y_3} A_1 = 0$. Substituting the expressions \eqref{Y2trunc} and \eqref{Y3trunc}, this condition reduces to:
	\begin{equation}
		\label{Y_consistency}
		(Y_2 A_2^{-1} + A_1 Y_2 A_2^{-2}) Y_3 A_1 = 0.
	\end{equation}
	Furthermore, we observe the fourth equation of \eqref{XYsystem}:
	$$ A_2 X_4 A_2 = X_4 A_2 X_4 + \widehat{Y_3} A_1 \widehat{Y_2}. $$
	Since $X_4$ commutes with $A_2$, the left-hand side is $X_4 A_2^2$ and the first term on the right is $X_4^2 A_2$. Evaluating the product $\widehat{Y_3} A_1 \widehat{Y_2}$ using \eqref{Y2trunc} and \eqref{Y3trunc}, and noting that terms containing $A_1^2$ vanish, we find:
	$$ \widehat{Y_3} A_1 \widehat{Y_2} = (A_2^{-1} Y_3 A_1) (-Y_2 A_2^{-1} - A_1 Y_2 A_2^{-2}) = -A_2^{-1} Y_3 A_1 Y_2 A_2^{-1}. $$
	Substituting this back into the main equation yields:
	$$ X_4 A_2^2 - X_4^2 A_2 = -A_2^{-1} Y_3 A_1 Y_2 A_2^{-1}. $$
	Since $A_2$ is invertible, we can multiply by $A_2$ on both sides and rearrange to obtain a quadratic equation for $X_4$:
	\begin{equation}
		\label{X4quadratic}
		X_4(A_2 - X_4)A_2^3 = -Y_3 A_1 Y_2.
	\end{equation}
	A necessary condition for the existence of $X_4$ in \eqref{X4quadratic} is that the matrix $-Y_3 A_1 Y_2$ must commute with $A_2$, as the left-hand side clearly does.
	
	Finally, the cross-terms in \eqref{XYsystem} provide linear coupling constraints. Pre-multiplying the second equation of \eqref{XYsystem} by $A_1$ (and using $A_1 X_1 A_1 = 0$) yields $A_1 Y_2 X_4 = 0$. Similarly, post-multiplying the third equation by $A_1$ yields $X_4 Y_3 A_1 = 0$. These conditions simplify the determination of the off-diagonal interactions.
	
	We summarize these findings in the following theorem:
	
	\begin{theorem}
		Let $A$ satisfy $A^n=tA^2$ with decomposition \eqref{cnA}, and let $Y$ be an anti-diagonal matrix $(Y_1=0, Y_4=0)$. If the system \eqref{XYsystem} is consistent then the following conditions are true:
		\begin{enumerate}
			\item \textbf{Conditions on $Y$:} The matrices $Y_2$ and $Y_3$ must satisfy:
			\begin{align}
				(Y_2 A_2^{-1} + A_1 Y_2 A_2^{-2}) Y_3 A_1 &= 0, \\
				Y_3 A_1 Y_2 A_2 &= A_2 Y_3 A_1 Y_2.
			\end{align}
			\item \textbf{Conditions on $X$:} There exist matrices $X_1$ and $X_4$ such that:
			\begin{align}
				X_1 A_1 &= A_1 X_1, \quad X_4 A_2 = A_2 X_4, \\
				A_1 Y_2 X_4 &= 0, \quad X_4 Y_3 A_1 = 0,
			\end{align}
			and they satisfy the algebraic equations:
			\begin{align}
				X_4(A_2 - X_4)A_2^3 &= -Y_3 A_1 Y_2, \label{quad_final}\\
				X_1^2 A_1 &= -\widehat{Y_2} A_2 \widehat{Y_3}, \\
				Y_2 X_4 &= A_1 Y_2 - X_1 A_1 Y_2 A_2^{-1}, \\
				X_4 Y_3 &= Y_3 A_1 - A_2^{-1} Y_3 A_1 X_1.
			\end{align}
		\end{enumerate}
	\end{theorem}

	\section{Worked examples}

	\begin{example}[Non-homogeneous YBME]\label{Ex3}
		Let $A=J_2(0)\oplus J_1(0),
		\;
		B=[b_{ij}],$ $X=[x_{ij}]\in M_3(\mathbb C).$
		A direct computation gives
		$AXA-XAX=B$ is equivalent to
		\begin{equation}\label{E2}
			\begin{aligned}
				&b_{11}=-x_{11}x_{21},\quad
				b_{12}=x_{21}-x_{11}x_{22},\quad
				b_{13}=-x_{11}x_{23},\\
				&b_{21}=-x_{21}^2,\quad
				b_{22}=-x_{21}x_{22},\quad
				b_{23}=-x_{21}x_{23},\\
				&b_{31}=-x_{31}x_{21},\quad
				b_{32}=-x_{31}x_{22},\quad
				b_{33}=-x_{31}x_{23}.
			\end{aligned}
		\end{equation}
		
		\medskip
		\textbf{1) Consistency relations on $B$.}
		
		\emph{Case A: $b_{21}\neq 0$.} From the equations for $b_{11}, b_{12}, b_{21}, b_{22},$ we will elliminate $x_{11}$ and $x_{22}$ using $x_{21}$. Since $b_{21}\neq 0$, we will take $x_{21}$ to be the square root of $-b_{21}$. From the first two equations, $x_{11}=-b_{11}/x_{21}$ and $x_{22}=-b_{22}/x_{21}$, and consequently $b_{12}=x_{21}-b_{11}b_{22}/x_{21}^2$. Using $x_{21}^2=-b_{21}$ and rearranging, we arrive at
		$x_{21}=b_{12}-b_{11}b_{22}/b_{21}$.
		We will call this quantity $s$, i.e.
		\begin{equation}\label{E3}
			s:= b_{12}-\frac{b_{11}b_{22}}{b_{21}}.
		\end{equation}
		Then \(AXA-XAX=B\) is consistent if and only if
		\begin{equation}\label{E4} 
			\begin{aligned}
				s^2=-\,b_{21},\;
				b_{13}=\frac{b_{11}b_{23}}{b_{21}},\;
				b_{32}=\frac{b_{31}b_{22}}{b_{21}},\; 
				b_{33}=\frac{b_{31}b_{23}}{b_{21}}.
			\end{aligned}
		\end{equation}
		
		\emph{Case B: $b_{21}=0$.}
		Then consistency is equivalent to
		\begin{equation}\label{E5}
			b_{11}=b_{21}=b_{22}=b_{23}=b_{31}=0,
		\end{equation}
		while the remaining four equations may be rewritten as
		$$
		\left[\begin{array}{cc}
			b_{12} & b_{13}\\
			b_{32} & b_{33}
		\end{array}\right]=-\left[\begin{array}{c}
			x_{11}\\
			x_{31}
		\end{array}\right][x_{22}\;x_{23}].
		$$
		Since any matrix of the form $uv^T$ is of the rank at most $1$, the rank of 
		$$\left[\begin{array}{cc}
			b_{12} & b_{13}\\
			b_{32} & b_{33}
		\end{array}\right]$$
		is at most $1$. It is a $2\times 2$ matrix, hence it must be
		\begin{equation}\label{E6}
			b_{12}b_{33}-b_{13}b_{32}=0.
		\end{equation}
		
		\medskip
		\textbf{2) All solutions $X$.}
		
		\emph{Case A: $b_{21}\neq 0$ and \eqref{E4} holds.}
		Let $s$ be as in \eqref{E3}. Then all solutions are given by
		\begin{equation}\label{E7}
			X=
			\begin{bmatrix}
				-\dfrac{b_{11}}{s} & \alpha & \beta\\[6pt]
				s & -\dfrac{b_{22}}{s} & -\dfrac{b_{23}}{s}\\[6pt]
				-\dfrac{b_{31}}{s} & \gamma & \delta
			\end{bmatrix},
			\qquad
			\alpha,\beta,\gamma,\delta\in\mathbb C.
		\end{equation}
		
		\emph{Case B: $b_{21}=0$ and \eqref{E5}--\eqref{E6} hold.}
		Choose any factorization
		\begin{equation}\label{E8}
			\begin{bmatrix}
				b_{12}&b_{13}\\
				b_{32}&b_{33}
			\end{bmatrix}
			=
			-\,\begin{bmatrix}u\\ v\end{bmatrix}
			\begin{bmatrix}p&q\end{bmatrix},
		\end{equation}
		and set
		\begin{equation}\label{E9}
			x_{21}=0,\quad
			x_{11}=u,\quad
			x_{31}=v,\quad
			x_{22}=p,\quad
			x_{23}=q,
		\end{equation}
		with $x_{12},x_{13},x_{32},x_{33}$ arbitrary.  
		This describes all solutions in Case B. \hfill$\clubsuit$
	\end{example}
	
	\begin{example}[YBME with given commutator]\label{Ex4}
		Let us consider the square-periodic matrix $A=\operatorname{diag}(J_2(0), J_1(0),i,-i)$ satisfying $A^4=-A^2$ (i.e. $n=4,$ $t=-1$). We have $N_1=3, N_2=2$ and the core-nilpotent decomposition is $A_1=J_2(0)\oplus 0, A_2=i\oplus -i$. We want to describe the class of all solutions $X$ for given matrix $Y$ with the following block-matrix forms:
		$$
		Y=\left[\begin{array}{cc}
			Y_1 & Y_2\\
			Y_3 & Y_4
		\end{array}\right]=\left[\begin{array}{ccc|cc}
			0 & 0 & 0 & -i & i\\
			0 & 0 & 0 & 0 & 0\\
			0 & 0 & 0 & 0 & 0\\
			\hline
			0 & i & 0 & 0 & 0\\
			0 & -1 & 0 & 0 & 0
		\end{array}\right],\;
		X=\left[\begin{array}{cc}
			X_1 & X_2\\
			X_3 & X_4
		\end{array}\right]=[x_{ij}]\in\mathbb{C}^{5\times 5}.
		$$
		
		First, using \eqref{X2series} and \eqref{X3series}, we see that for those $Y_2$ and $Y_3$ we obtain the unique matrices $X_2$ and $X_3$ (denoted by $\widehat{Y_2}$ and $\widehat{Y_3}$ in the paper):
		$$
		\widehat{X_2}=\left[\begin{array}{cc}
			1 & 1\\
			0 & 0\\
			0 & 0
		\end{array}\right],\;
		\widehat{X_3}=\left[\begin{array}{ccc}
			0 & 1 & 0\\
			0 & -i & 0
		\end{array}\right].
		$$
		
		Theorem \ref{Th3.2.} together with Example \ref{Ex1} describes the set of all admissible $X_1$ such that $A_1X_1-X_1A_1=Y_1=0_{3\times 3}$:
		$$X_1^\sigma =\left[\begin{array}{ccc}
			x_{11} & x_{12} & x_{13}\\
			0 & x_{11} & 0\\
			0 & x_{32} & x_{33}
		\end{array}\right].$$
		
		Theorem \ref{Th4.1.} describes the set of all admissible $X_4$ such that $A_2X_4-X_4A_2=Y_4=0_{2\times 2}$: 
		$$X_4^\sigma =\left[\begin{array}{cc}
			x_{44} & 0\\
			0 & x_{55}
		\end{array}\right].$$
		Therefore, so far we have the following wide class that contains all solutions (and many non-solutions) of $AX-XA=Y$:
		$$X^\sigma=\left[\begin{array}{ccc|cc}
			x_{11} & x_{12} & x_{13} & 1 & 1\\
			0 & x_{11} & 0 & 0 & 0\\
			0 & x_{32} & x_{33} & 0 & 0\\
			\hline
			0 & 1 & 0 & x_{44} & 0\\
			0 & -i & 0 & 0 & x_{55}
		\end{array}\right].$$
		So far we solved $AX-XA=Y$ completely, and now we want to address the following system, bearing in mind that $A_1X_1-X_1A_1=Y_1$ and $A_2X_4-X_4A_2=Y_4$:
		\begin{align}
			A_1 X_1 A_1 - X_1 A_1 X_1 &= \widehat{Y_2} A_2\widehat{Y_3}\\
			A_2 X_4 A_2 - X_4 A_2 X_4 &= \widehat{Y_3} A_1 \widehat{Y_2},\\
			X_1 A_1 \widehat{Y_2}+\widehat{Y_2} A_2 X_4 &= A_1 \widehat{Y_2} A_2,\\
			\widehat{Y_3} A_1 X_1+X_4 A_2\widehat{Y_3} &= A_2\widehat{Y_3} A_1.
		\end{align}
		We will now investigate the consistency of the system \eqref{secthird}, that is, the last two equations; then we will consider the first and the second one, and the final solution set will be their intersection. Since
		$$\A=\left[
		\begin{array}{ccccccccccccc}
			0 & 0 & 0 & 0 & 0 & 0 & 0 & 0 & 0 & i & -i & 0 & 0 \\
			0 & 0 & 0 & 0 & 0 & 0 & 0 & 0 & 0 & 0 & 0 & 0 & 0 \\
			0 & 0 & 0 & 0 & 0 & 0 & 0 & 0 & 0 & 0 & 0 & 0 & 0 \\
			0 & 0 & 0 & 0 & 0 & 0 & 0 & 0 & 0 & 0 & 0 & i & -i \\
			0 & 0 & 0 & 0 & 0 & 0 & 0 & 0 & 0 & 0 & 0 & 0 & 0 \\
			0 & 0 & 0 & 0 & 0 & 0 & 0 & 0 & 0 & 0 & 0 & 0 & 0 \\
			0 & 0 & 0 & 0 & 0 & 0 & 0 & 0 & 0 & 0 & 0 & 0 & 0 \\
			0 & 0 & 0 & 0 & 0 & 0 & 0 & 0 & 0 & 0 & 0 & 0 & 0 \\
			0 & 0 & 0 & 0 & 0 & 0 & 0 & 0 & 0 & i & 0 & -1 & 0 \\
			0 & 0 & 0 & 0 & 0 & 0 & 0 & 0 & 0 & 0 & i & 0 & -1 \\
			0 & 0 & 0 & 0 & 0 & 0 & 0 & 0 & 0 & 0 & 0 & 0 & 0 \\
			0 & 0 & 0 & 0 & 0 & 0 & 0 & 0 & 0 & 0 & 0 & 0 & 0 \\
		\end{array}
		\right]_{12\times 13}$$
		and $A_1\widehat{X_2}A_2=0_{3\times 2},\;A_2\widehat{X_3}A_1=0_{2\times 3}$, vectorization gives $\mathbf{b}=0_{1\times 12}$, so the system
		$\A\mathbf{x}=\mathbf{b}$ with $\mathbf{x}=\{z_1, z_2, z_3, z_4, z_5, z_6, z_7, z_8, z_9, z_{10}, z_{11}, z_{12}, z_{13}\}$ is consistent iff $\A\A^\dag \mathbf{b}=\mathbf{b}$ (which is true) and all solutions are given by
		\begin{align*} 
			\mathbf{x}&=\left[z_1, z_2, z_3, z_4, z_5, z_6, z_7, z_8, z_9, \frac{z_{10} + z_{11}- i z_{12}-i z_{13}}{4},\right. \\
			&\phantom{\qquad}  \left.\frac{z_{10} + z_{11}- i z_{12}-i z_{13}}{4}, \frac{iz_{10} + iz_{11}+z_{12}+ z_{13}}{4}, \frac{iz_{10} + iz_{11}+z_{12}+ z_{13}}{4}\right]^T
		\end{align*} 
		and the de-vectorization leads to
		$$X''=\left[
		\begin{array}{ccc|cc}
			z_1 & z_4 & z_7 & 1 & 1 \\
			z_2 & z_5 & z_8 & 0 & 0 \\
			z_3 & z_6 & z_9 & 0 & 0 \\
			\hline 
			0 & 1 & 0 & \frac{z_{10}}{4}+\frac{z_{11}}{4}-\frac{i z_{12}}{4}-\frac{i z_{13}}{4} & \frac{i z_{10}}{4}+\frac{i z_{11}}{4}+\frac{z_{12}}{4}+\frac{z_{13}}{4} \\
			0 & -i & 0 & \frac{z_{10}}{4}+\frac{z_{11}}{4}-\frac{i z_{12}}{4}-\frac{i z_{13}}{4} & \frac{i z_{10}}{4}+\frac{i z_{11}}{4}+\frac{z_{12}}{4}+\frac{z_{13}}{4} \\
		\end{array}
		\right].$$

		Let us now consider the non-homogeneous YBME $$A_1\widehat{X_1}A_1-\widehat{X_1}A_1\widehat{X_1}=B_{12}=\widehat{Y_2}A_2\widehat{Y_3}=\left[\begin{array}{ccc}
			0 & i-1 & 0\\
			0 & 0 & 0\\
			0 & 0 & 0
		\end{array}\right].$$
		using Theorem \ref{inhomYBMEsol}. We read:
		$$M=[0\;x_{11}\;0],\;R_2=R_3=0_{1\times 3},\;R_1(0)=[0\;1-i\;0].$$
		Also, $M^-=[\alpha\;1/x_{11}\;\beta]^T$ for any complex $\alpha,\beta$, implying $MM^-=[1]$ and $M^-M=\left[\begin{array}{ccc}
			0 & \alpha x_{11} & 0\\
			0 & 1 & 0\\
			0 & \beta x_{11} & 0
		\end{array}\right]$. The equation is consistent because $x_{21}^2=-b_{21}(=0)$ holds, as well as \eqref{cond2} and \eqref{cond3}. Now, to describe the solutions, we follow the procedure outlined in the theorem:
		\begin{itemize}
			\item[i)] $x_{21}=0$
			\item[ii)] $x_{22}=ub_{22}+w_{22}=w_{22},\;x_{23}=vb_{23}+w_{23}=w_{23}$ for any complex $u,v$
			\item[iii)] $x_{31}=0+w_{31}(1-MM^-)=0$ and
			$$x_{11}=[0\;1-i\;0]\left[\begin{array}{c}
				\alpha\\
				1/x_{11}\\
				\beta 
			\end{array}\right]+w_{11}(1-MM^-)=\frac{1-i}{x_{11}}.$$
		\end{itemize}
		This part iii) implies $x_{11}^2=1-i$, hence $x_{11}=\pm\sqrt{1-i}$, while from earlier we know that $x_{22}=x_{11}$ so $w_{22}$ is not arbitrary at all although ii) may mislead the reader. From similar reasons $x_{23}=0$ not arbitrary $w_{23}$. Therefore, so far we know that
		$$\widetilde{X_1}=\left[\begin{array}{ccc}
			\pm\sqrt{1-i} & w_{12} & w_{13}\\
			0 & \pm\sqrt{1-i} & 0\\
			0 & w_{23} & w_{33},
		\end{array}\right]$$
		where $w_{12}, w_{13}, w_{32}, w_{33}$ are arbitrary complex numbers.
		
		Now we will solve the homogeneous YBME
		$$A_2\widehat{X_4}A_2-\widehat{X_4}A_2\widehat{X_4}=\widehat{Y_3}A_1\widehat{Y_2}=0_{2\times 2}.$$
		Immediate calculation shows that there are precisely four solutions, two trivial ($0$ and $A_2$), and two spectral:
		$$A_2P_i=\left[\begin{array}{cc}
			i & 0\\
			0 & 0
		\end{array}\right],\;A_2P_{-i}=\left[\begin{array}{cc}
			0 & 0\\
			0 & -i
		\end{array}\right].$$
		In the intersection of those two sets of solutions, we conclude that it must be $X_4=0_{2\times 2}$. Indeed, the $X_4-$block obtained via $\A$ is of the form
		$$\left[\begin{array}{cc}
			h & ih\\
			h & ih
		\end{array}\right]=h\left[\begin{array}{cc}
			1 & i\\
			1 & i
		\end{array}\right],$$
		and the only matrix among $0, A_2, A_2P_i, A_2P_{-i}$ that can be written in this form is $0$.
		Therefore, the solutions to the original YBME are given by
		$$X=\left[\begin{array}{ccc|cc}
			\pm\sqrt{1-i} & x_{12} & x_{13} & 1 & 1\\
			0 & \pm\sqrt{1-i} & 0 & 0 & 0\\
			0 & x_{32} & x_{33} & 0 & 0\\
			\hline
			0 & 1 & 0 & 0 & 0\\
			0 & -i & 0 & 0 & 0
		\end{array}\right],$$
		and they depend on four parameters $x_{12}, x_{13}, x_{32}, x_{33}\in\CC$. \hfill$\clubsuit$
	\end{example}
	
	\subsection*{Concluding remarks}
	In this paper we have shown how the employment of an \emph{a priori} provided commutator $Y\equiv AX-XA$ defines entire classes of new non-commuting solutions to the YBME $AXA=XAX$. Moreover, we have demonstrated that this approach is also valid in solving some special cases of the inhomogeneous YBME $AXA=XAX+B$, which, to the best of our knowledge, has not yet been studied. The obtained necessary conditions for the existence of such commutator matrices $Y$ are directly derived from the conditions imposed on the coefficient matrix $A$: in our paper, we assumed that $A^n=tA^2$, however, this idea would work for any polynomially-cyclic condition imposed on the matrix $A$. \\
	
	\noindent\textbf{Comment.} These results were obtained as a part of the Student Internship 2025 organized by the Mathematical Institute of the Serbian Academy of Sciences and Arts.\\
	
	\noindent\textbf{Conflict of interest.} The authors declare that there are no conflicts of interest in publishing the findings obtained in this paper.\\
	
	\noindent\textbf{Data availability statement.} Data availability is not applicable as no data sets were generated during this research.\\
	
	\noindent\textbf{Funding.} The first author is supported by the Ministry of Education, Science and Technological Development, Republic of Serbia, grant No. 451-03-33/2026-03/200029. The second author is supported by the Ministry of Education, Science and Technological Development, Republic of Serbia, grant No. 451-03-34/2026-03/200124.

	

\begin{thebibliography}{9}
		
		
		\bibitem{BiG} A. Ben-Israel, T. N. E. Greville, {\it Generalized inverses, theory and applications}, 2nd ed., Springer, 2003.
		
		\bibitem{DIN} N. \v C. Din\v ci\'c, {\it Solving the Sylvester equation $AX-XB = C$ when $\sigma(A)\cap \sigma(B) \neq\emptyset$}, Electron. J. Linear Algebra 35 (2019), 1--23.
		
		\bibitem{DinDjor} N. \v C. Din\v ci\'c and B. D. Djordjevi\'c, {\it On the intrinsic structure of the solution set to the Yang-Baxter-like matrix equation}, Rev. Real Acad. Cienc. Exactas Fis. Nat. Ser. A-Mat. 116:73 (2022).
		
		
		\bibitem{NCDBDD2}N. \v C. Din\v ci\'c and B. D. Djordjevi\'c, {\it Yang-Baxter-like matrix equation: a road less taken}. In Matrix and Operator
		Equations (Mathematics Online First Collections), 2023.
		
		\bibitem{BDperm} B. D. Djordjevi\'c, {\it Doubly stochastic and permutation solutions to $AXA=XAX$ when $A$ is a permutation matrix}, Linear Algebra Appl. 661 (2023), 79--105.
		
		
		\bibitem{BDD} B. D. Djordjevi\'c, {\it The equation $AX-XB=C$ without a unique solution: the ambiguity which benefits applications}, Zb. Rad. (Beogr.) 2022.
		
		\bibitem{BDND1} B. D. Djordjevi\'c and  N. \v C. Din\v ci\'c, {\it Classification and approximation of solutions to Sylvester matrix equation}, Filomat 33 (13) (2019), 4261--4280.\\
		https://doi.org/10.2298/FIL1913261D
		
		
		\bibitem{Don2016} Q. Dong, {\it Projection-based commuting solutions of the Yang–Baxter
			matrix equation for non-semisimple eigenvalues}, Appl. Math. Lett. 64 (2017), 231--234
		
		\bibitem{DonDin2017} Q. Dong and J. Ding, {\it Complete commuting solutions of the Yang–Baxter-like matrix equation for diagonalizable matrices}, Computers and Mathematics with Applications 72:1 (2016), 194--201
		
		\bibitem{Dong2018} Q. Dong, J. Ding, Q. Huang, {\it Commuting solutions of a quadratic matrix equation for nilpotent matrices}, Algebra Colloquium 25 (2018), 31-44.\ 10.1142/S1005386718000032
		
		
		\bibitem{Higham} N. J. Higham, {\it Functions of matrices. Theory and computation}, SIAM, 2008.
		
		\bibitem{QH} Q. Huang, M. Saeed Ibrahim Adam, J. Ding, L. Zhu, {\it All non-commuting solutions of the Yang-Baxter matrix equation for a class of diagonalizable matrices}, Oper. Matrices 13 (1) (2019) 87--195.
		
		\bibitem{Pen} R. Penrose, {\it A generalized inverse for matrices}, Proc. Cambridge Philos. Soc. 51 (1955), 406--413.
		
		
		\bibitem{MSIA} M. Saeed Ibrahim Adam, J. Ding, and Q. Huang, {\it Explicit solutions of the Yang-Baxter like matrix equation for an idempotent matrix}, Appl. Math. Lett. 63 (2017), 71--76.
		
		\bibitem{MSIAJDQ} M. Saeed Ibrahim Adam, J. Ding, Q. Huang and L. Zhu, {\it Solving a class of quadratic matrix equations}, Appl. Math. Lett. 82 (2018), 58--63.
		
		\bibitem{MSIAJDQHLZ} M. Saeed Ibrahim Adam, J. Ding, Q. Huang and L. Zhu, {\it All solutions of the Yang-Baxter-like matrix equation when $A^3=A$}, Journal Appl. Anal. Comp. 9:3 (2019) 1022--1031\\ 
		doi: 10.11948/2156-907X.20180244
		
		
		\bibitem{DZJD} D. Zhou and J. Ding, {\it All solutions of the Yang-Baxter-like matrix equation for nilpotent matrices of index two}, Complexity. 2020. 1-7.\\ 10.1155/2020/2585602.
		
		%
		\bibitem{a4a} D. Zhou, J. Liao Y. Gan H. Xu, and R. Zhang, \textit{All solutions of the Yang-Baxter-like matrix equation $AXA=XAX$ with $A$ satisfying $A^4=A$}, J. Math. Anal. Appl. 552 (2025) 129785\\ https://doi.org/10.1016/j.jmaa.2025.129785
		
		\bibitem{ZYWDH2021} D. Zhou, X. Ye, Q. Wang, J. Ding, and W. Hu, {\it Explicit solutions of the Yang-Baxter-like matrix equation for a singular diagonalizable matrix with three distinct eigenvalues}, Filomat 35 (2021) 3971-3982,\\
		10.2298/FIL2112971Z
		
		
	\end{thebibliography}
\end{document}